\documentclass[12pt, reqno]{amsart}
\usepackage{amsmath, amstext, amsbsy, amssymb, ytableau}

\makeatletter
\@namedef{subjclassname@2020}{%
  \textup{2020} Mathematics Subject Classification}
\makeatother

\newtheorem{theorem}{Theorem}[section] 
\newtheorem{lemma}[theorem]{Lemma}
\newtheorem{proposition}[theorem]{Proposition}

\theoremstyle{definition}
\newtheorem{definition}[theorem]{Definition}
\newtheorem{example}[theorem]{Example}

\theoremstyle{remark}
\newtheorem{remark}[theorem]{Remark}
\theoremstyle{conjecture}
\newtheorem{conjecture}[theorem]{Conjecture}

\numberwithin{equation}{section}

\newcommand{\BD}{\mathcal {BD}}

\newcommand{\ch}{{\rm ch} }
\newcommand{\C}{ \mathbb C }
\newcommand{\Coe}{ {\rm Coeff} }

\newcommand{\End}{{\rm End}}
\newcommand{\fa}{ \mathfrak a }
\newcommand{\fB}{ \mathfrak B }

\newcommand{\fG}{ \mathfrak G }
\newcommand{\fL}{ \mathfrak L }

\newcommand{\fp}{ \mathfrak p }
\newcommand{\fock}{{\mathbb H}_X}

\newcommand{\Hn}{H^*(\Xn)}

\newcommand{\la}{{\lambda}}
\newcommand{\lambsq}{s(\lambda)}

\newcommand{\MD}{\mathcal {MD}}
\newcommand{\N}{ \mathbb N }
\newcommand{\On}{{\mathcal O}_X^{[n]}}

\newcommand{\Q}{ \mathbb Q }
\newcommand{\qBD}{{\rm q}{\mathcal {BD}}}
\newcommand{\qMD}{{\rm q}{\mathcal {MD}}}
\newcommand{\qMZV}{ {\bf qMZV} }

\newcommand{\T}{ \mathbb T }
\newcommand{\Tr}{ {\rm Tr} }

\newcommand{\vac}{|0\rangle}

\newcommand{\w}{\tilde}
\newcommand{\W}{\widetilde}
\newcommand{\Wb}{ {\bf W} }

\newcommand{\Xn}{ {X^{[n]}}}
\newcommand{\Z}{ \mathbb Z }

\def\beq{\begin{equation}}
\def\eeq{\end{equation}}

\numberwithin{equation}{section}

\begin{document}

\title[Equivariant Chern character operators and Okounkov's conjecture]
      {Equivariant Chern character operators and Okounkov's conjecture}

\author[Mazen M. Alhwaimel]{Mazen M. Alhwaimel}
\address{Department of Mathematics, College of Science, Qassim University, 
         P. O. Box 6644, Buraydah 51452, Saudi Arabia} 
\email{4200@qu.edu.sa}

\author[Zhenbo Qin]{Zhenbo Qin}
\address{Department of Mathematics, University of Missouri, 
         Columbia, MO 65211, USA} 
\email{qinz@missouri.edu}

\date{\today}

\begin{abstract}
In this paper, we study the Chern character operators 
on the equivariant cohomology of the Hilbert schemes of points 
in the complex affine plane $\C^2$ with the action of 
the torus $\T = (\C^*)^2$, and partially verify 
Okounkov's Conjecture \cite[Conjecture~2]{Oko} in this setting. 
Our main idea is to apply the connection between 
the equivariant cohomology of these Hilbert schemes and 
the ring of symmetric functions, via the deformed vertex operators of 
Cheng and Wang \cite{CW}, (the integral form of) 
the Jack symmetric functions and  
the transformed Macdonald symmetric functions \cite{GH, Hai}.
\end{abstract}

\maketitle

\section{\bf Introduction} 
\label{sect_intr}

It is well-known \cite{Bri, Fog, Iar} that Hilbert schemes of points on a smooth algebraic surface are irreducible and smooth. 
In \cite{Oko}, Okounkov introduced some reduced generating series 
for the intersection numbers among the Chern characters of 
the tautological bundles over the Hilbert schemes of points on 
a smooth algebraic surface and the total Chern classes of 
the tangent bundles of these Hilbert schemes, 
and conjectured that they are multiple $q$-zeta values which are $q$-deformations of the usual (multiple) zeta values 
\cite{Bac, BK1, BK2, Bra1, Bra2, OT, Zhao, Zud}. 
This conjecture has been investigated and partially verified in 
\cite{Alh, AQ, Car1, Car2, CO, Qin2, QY, Zhou}. 
The basic idea in these papers is to interpret 
the reduced series in terms of the Chern character operators 
\cite{LQW1} which act on the cohomology of 
the Hilbert schemes via cup products. The Chern character operators 
have played a pivotal role in determining the cohomology ring structures 
of the Hilbert schemes of points on surfaces 
\cite{Leh, LQW1, LQW2, LQW4, Qin1}. 

In this paper, we study the Chern character operators and 
Okounkov's Conjecture in the equivariant setting which generalizes 
the setting in \cite{Car1, Car2} and 
coincides with the setting in \cite{CO, Nak2, OP1, OP2, Zhou}. 
Let $X = \C^2$ be the complex affine plane and $\T = (\C^*)^2$. 
Let $u, v$ be the standard coordinate functions on $X$. 
Define the action of $\T$ on $X$ by 
\begin{eqnarray}  \label{Intro_TT08250207X}
(a, b) \cdot (u, v)=(a^{-1}u, b^{-1}v), \quad (a, b) \in \T.
\end{eqnarray}
The origin $x \in X = \mathbb C^2$ is the only fixed point. 
The action of $\T$ 
on $X$ lifts to an action of $\T$ on the Hilbert scheme $\Xn$ 
which parametrizes length-$n$ $0$-dimensional closed subschemes of $X$.
The $\T$-fixed points of $\Xn$ are supported on the origin $x \in X$ and
indexed by partitions $\lambda$ of $n$. We use $\xi_{\lambda}$
to denote the fixed point in $\Xn$ corresponding to
a partition $\lambda = (\la_1 \ge \la_2 \ge \cdots \ge \la_k > 0)$ of $n$, 
and adopt the convention that $\xi_{\lambda}$ is defined by the ideal 
\begin{eqnarray}  \label{Intro_VV06141023Z}
  I_{\xi_\la} 
= \big \langle u^{\la_1}, u^{\la_2}v, \ldots, u^{\la_k}v^{k-1},v^k \big \rangle.
\end{eqnarray}
Put $H^*_\T({\rm pt}) = \C[t_1, t_2]$ and $H^*_{\mathbb T}(\cdot)^{\prime} = H^*_{\mathbb T}(\cdot)
\otimes_{\mathbb C[t_1, t_2]} \mathbb C(t_1, t_2)$. Define
\begin{eqnarray} 
\mathbb H_X' = \bigoplus_{n=0}^{+\infty} H_\T^*(\Xn)^{\prime}. \label{Intro_def_HX'}
\end{eqnarray}
By the localization theorem, the equivariant classes $[\xi_\la]$ form 
a linear basis of $\mathbb H_X'$. 

The tautological rank-$n$ vector bundle $\On$ over $\Xn$ 
is $\T$-equivariant. The fiber of $\On$ over the fixed point 
$\xi_\la \in \Xn$ is canonically 
\begin{eqnarray}   \label{Intro_VV0614941Z}
\On |_{\xi_\lambda} = H^0 \big ( \mathcal O_{\xi_\la} \big ) 
= \C[u, v]/\big \langle u^{\la_1}, u^{\la_2}v, \ldots, 
  u^{\la_k}v^{k-1},v^k \big \rangle.
\end{eqnarray}
It follows that as $\T$-modules, we have
\begin{eqnarray} \label{Intro_TT08251244X}
\On |_{\xi_\lambda} 
=\bigoplus\limits_{\square\in D_\lambda}
 \theta_1^{-a'(\square)} \theta_2^{-\ell'(\square)}
\end{eqnarray}
where $\theta_1$ and $\theta_2$ are the two standard $1$-dimensional  
$\T$-modules, and $a'(\square)$ (respectively, $\ell'(\square)$) 
is the arm colength (respectively, the leg colength) of the cell $\square$ 
in the Young diagram $D_\la$ of the partition $\la$.

\begin{definition}   \label{Intro_DefFG}
Let $k \ge 0$. The {\it equivariant Chern character operator} 
$\fG(t_1, t_2) \in \End(\fock')$ 
and the {\it $k$-th equivariant Chern character operator} 
$\fG_k(t_1, t_2) \in \End(\fock')$ are defined by putting
\begin{eqnarray}  
   \fG(t_1, t_2) [\xi_\la] 
= \ch^\T\big ( \mathcal O_X^{[|\la|]} \big ) \cup [\xi_\la]
     = \sum_{\square\in D_\lambda} 
     e^{-a'(\square)t_1 - \ell'(\square)t_2} \cdot [\xi_\la],  
     \qquad\qquad \label{Intro_DefFG.1}   \\
   \fG_k(t_1, t_2) [\xi_\la] 
= \ch_k^\T\big ( \mathcal O_X^{[|\la|]} \big ) \cup [\xi_\la]
     = \sum_{\square\in D_\lambda} \frac{(-1)^k}{k!} 
     (a'(\square)t_1 + \ell'(\square)t_2)^k \cdot [\xi_\la]  
     \quad \label{Intro_DefFG.2}
\end{eqnarray}
for every partition $\la$, and by putting
$\fG(t_1, t_2)|_{H_\T^*(X^{[0]})^{\prime}} = 0 = 
\fG_k(t_1, t_2)|_{H_\T^*(X^{[0]})^{\prime}}$. 
\end{definition}

Our goal is to study the equivariant Chern character operators 
$\fG_k(t_1, t_2) \in \End(\fock')$. 
It is known \cite{LQW4, OP2, Qin1, Vas} that $\fock'$ is 
linearly isomorphic to $\mathfrak F \otimes_\C \C(t_1, t_2)$ 
where $\mathfrak F$ is the Fock space  
freely generated over $\C$ by commuting creation operators 
$\fa_{-k}, k \in \Z_{>0}$ acting on the vacuum vector $\vac$. 
The annihilation operators $\fa_{k}, k \in \Z_{>0}$ satisfy 
$\fa_{k} \vac = 0$, and obey the commutation relations
\begin{eqnarray}  \label{Intro_TT09110234X}
[\fa_k, \fa_{\ell}] = k \delta_{k, -\ell} \cdot \text{\rm Id}.
\end{eqnarray}
The Fock space $\mathfrak F$ can be further identified with 
$\Lambda_\C$ which denotes the ring (over $\C$) of symmetric functions 
in countably many independent variables $x_1, x_2, \ldots$. 
The following deformed vertex operator is introduced in \cite{CW}:
\begin{eqnarray}  \label{Intro_CW1}
   V(z; q, t, \w q, \w t) 
= \exp\left (\sum_{k \ge 1} (q^k - \w q^k) \frac{z^k}{k}\fa_{-k} \right )
  \exp\left (\sum_{k \ge 1} (\w t^k - t^k) \frac{z^{-k}}{k}\fa_k \right ).
\end{eqnarray}
Let $V_0(z; q, t, \w q, \w t) = \Coe_{z^0} V(z; q, t, \w q, \w t)$ 
be the $0$-mode of $V(z; q, t, \w q, \w t)$, and define
\begin{eqnarray}  \label{Intro_TT10180917Z.0}
  \overline{\fB}(q, t^{-1})
= \frac{1}{(1-q)(1-t^{-1})} \big (1 - V_0(z; t, t^{-1}, 1, qt^{-1})\big ).
\end{eqnarray}
Let $q = e^{\alpha t_0}$ where $\alpha$ is a formal parameter, 
and $t = e^{t_0}$. For $k \ge 0$, let 
\begin{eqnarray}  \label{Intro_OverlinefBkAlpha}
\overline{\fB}_k^{(\alpha)} = \Coe_{t_0^k} \overline{\fB}(q, t^{-1}). 
\end{eqnarray}

\begin{theorem}    \label{Intro_TT10210727X}
If $0 \le k \le 2$, then 
\begin{eqnarray}    \label{Intro_TT10210727X.0}
  \fG_k(t_1, t_2)
= t_2^k t_1^{\delta(\cdot)} 
  \overline{\fB}_k^{(\alpha)}|_{\alpha = -t_1/t_2}
\end{eqnarray}
where $t_1^{\delta(\cdot)} \fa_\la = t_1^{\delta(\la)} \fa_\la 
= t_1^{\ell(\la^-) - \ell(\la^+)} \fa_\la$
for a generalized partition $\la$, and $\la^+$ and $\la^-$ are 
the positive and negative parts of $\la$ respectively.
\end{theorem}

We refer to Definition~\ref{partition}~(i) for the concepts and notations 
related to generalized partitions. When $t_1 = -t_2 = 1$, the operators 
$\fG_k(t_1, t_2)$ have been determined in \cite{LQW3, Car1, Car2}.
When $k = 0$, formula \eqref{Intro_TT10210727X.0} 
yields the trivial identity 
\begin{eqnarray}    \label{Intro_G0t1t2}
  \fG_0(t_1, t_2)
= \sum_{i>0} \fa_{-i} \fa_{i}.
\end{eqnarray}
When $k = 1$, formula \eqref{Intro_TT10210727X.0} recovers 
\begin{eqnarray}    \label{Intro_OP3DefMqt1t2}
  \fG_1(t_1, t_2)
= \frac{1}{2} \sum_{i, j>0}  (t_1t_2 \fa_{-i}\fa_{-j}\fa_{i+ j}
  - \fa_{-i - j}\fa_{i}\fa_j)
  - (t_1+t_2)  \sum_{i>0} \frac{i-1}{2} \fa_{-i} \fa_{i}
\end{eqnarray} 
which has been obtained by \cite[Theorem~1]{OP1} (setting $q = 0$ there) 
and is known as the equivariant boundary operator following \cite{Leh}. 
Using \eqref{Intro_TT10210727X.0}, we obtain an explicit formula for 
$\fG_2(t_1, t_2)$ in \eqref{BB02241101Z}.
For $k \ge 3$, it is unlikely that \eqref{Intro_TT10210727X.0} holds. 
On the other hand, a direct computation reveals that 
$t_2^k t_1^{\delta(\cdot)} 
\overline{\fB}_k^{(\alpha)}|_{\alpha = -t_1/t_2}$ equals
$$
\sum_{\ell(\la) = k + 2, \, |\la| = 0} (-1)^{\ell(\la^+)-1} 
(t_1t_2)^{\ell(\la^-)-1} \frac{\fa_\la}{\la^!} 
+ \sum_{\ell(\la) < k + 2, \, |\la| = 0} b^{(k)}_\la \frac{\fa_\la}{\la^!}
$$
where $b^{(k)}_\la \in \Q(t_1, t_2)$. 
Together with the formula of the Chern character operators in 
the non-equivariant setting (see Theorem~\ref{char_th}), we propose 
the following conjecture regarding the leading term of $\fG_k(t_1, t_2)$ 
(in other words, our conjecture asserts that $\fG_k(t_1, t_2)$ and 
$t_2^k t_1^{\delta(\cdot)} 
\overline{\fB}_k^{(\alpha)}|_{\alpha = -t_1/t_2}$ 
have the same leading term).

\begin{conjecture}  \label{Intro_Conj_EquivGk}
Let $k \ge 0$. Then, $\fG_k(t_1, t_2)$ is equal to
\begin{eqnarray}   \label{Intro_Conj_EquivGk.01}
\sum_{\ell(\la) = k + 2, \, |\la| = 0} (-1)^{\ell(\la^+)-1} 
(t_1t_2)^{\ell(\la^-)-1} \frac{\fa_\la}{\la^!} 
+ \sum_{\ell(\la) < k + 2, \, |\la| = 0} g^{(k)}_\la \frac{\fa_\la}{\la^!}
\end{eqnarray} 
where $g^{(k)}_\la \in \Q(t_1, t_2)$, and $\la^+$ and $\la^-$ are 
the positive and negative parts of $\la$. 
\end{conjecture} 

The main idea in the proof of Theorem~\ref{Intro_TT10210727X} is 
to exploit the connection between the Fock space $\fock'$ and 
the ring $\Lambda_\C$ of symmetric functions. Using the transformed 
Macdonald symmetric functions \cite{CW, GH, Hai, Mac}, we show in 
Lemma~\ref{TT10180917Z} that the integral form $J_\la(X; q, t)$ of 
the Macdonald symmetric function is an eigenfunction 
of the operator $\overline{\fB}(q, t^{-1})$ with eigenvalue 
$\sum_{\square\in D_\lambda} q^{a'(\square)}t^{-\ell'(\square)}$.

An important feature of Theorem~\ref{Intro_TT10210727X} is the connection 
between the equivariant Chern character operators $\fG_k(t_1, t_2)$ 
where $0 \le k \le 2$ and vertex operators through \eqref{Intro_CW1}, 
\eqref{Intro_TT10180917Z.0} and \eqref{Intro_OverlinefBkAlpha}. 
It leads to two interesting applications. The first is to use 
Theorem~\ref{Intro_TT10210727X} to investigate 
Okounkov's Conjecture in our equivariant setting. 
Okounkov's reduced generating series is defined by
\begin{eqnarray}   
  \langle \ch_{k_1} \cdots \ch_{k_N} \rangle_{m, t_1, t_2}'
= \frac{\langle \ch_{k_1} \cdots \ch_{k_N} 
   \rangle_{m, t_1, t_2}}{\langle \rangle_{m, t_1, t_2}}       
   \label{Intro_Reduced.1}
\end{eqnarray}
where $k_1, \ldots, k_N \ge 0$, $m \in \Z$, and via localization 
(see Remark~\ref{rmk_trace}~(ii) and (i)), 
\begin{eqnarray}   
   {\langle \ch_{k_1} \cdots \ch_{k_N} \rangle_{m, t_1, t_2}}
= \Tr_{\fock'} q^{\mathfrak n} \, 
       \Gamma_-(1)^{m+t_1+t_2} \, \Gamma_+(1)^{m/(t_1t_2)} \prod_{j=1}^N 
       \fG_{k_j}(t_1, t_2).     \label{Intro_trace.0100}   
\end{eqnarray}
In \eqref{Intro_trace.0100}, $\mathfrak n$ is 
the number-of-points operator 
(i.e., $\mathfrak n|_{H^*(X^{[n]})'_\T} =  n \, {\rm Id}$), and 
\beq   \label{Intro_GammaLz}
\Gamma_\pm(z) 
= \exp \left ( \sum_{k > 0} \frac{z^{\mp k}}{k} \fa_{\pm k} \right ).
\eeq
\cite[Conjecture~2]{Oko} states that 
$\langle \ch_{k_1} \cdots \ch_{k_N} \rangle_{m, t_1, t_2}'$ is 
a multiple $q$-zeta value in $\qMZV$ of weight at most 
$\sum_{i=1}^N (k_i+2)$. We refer to Section~\ref{sect_qMZV} for 
the definitions and notations of multiple $q$-zeta values 
and the $\Q$-algebras $\qMZV \subset \qBD \subset \BD$. 
Our next result partially verifies Okounkov's Conjecture in 
the equivariant setting.

\begin{theorem}     \label{Intro_BB070340X}
\begin{enumerate}
\item[{\rm (i)}]
Let $k_1, \ldots, k_N$ be nonnegative integers. Then, 
$$
\langle \ch_{k_1} \cdots \ch_{k_N} \rangle_{0, t_1, t_2}' 
\in \qBD[t_1, t_2]
$$
is a degree-$\sum_{i=1}^N k_i$ symmetric homogeneous polynomial in 
$t_1$ and $t_2$ whose coefficients have weights at most 
$\sum_{i=1}^N (k_i+2)$.

\item[{\rm (ii)}]
Let $k \ge 0$. Then, Okounkov's Conjecture \cite[Conjecture~2]{Oko} 
holds for $\langle \ch_{k} \rangle_{0, t_1, t_2}'$.
More precisely, $\langle \ch_{k} \rangle_{0, t_1, t_2}' 
\in \qMZV[t_1, t_2]$ is a degree-$k$ symmetric homogeneous polynomial in 
$t_1$ and $t_2$ whose coefficients have weights at most 
$(k+2)$, and the coefficient of $t_1^k$ (and $t_2^k$) in 
$\langle \ch_{k} \rangle_{0, t_1, t_2}'$ is equal to 
$$
(-1)^k \cdot [k+2] + \big (\text{terms with weights $< (k+2)$} \big ).
$$

\item[{\rm (iii)}] 
If $k_1, \ldots, k_N \in \{0, 1, 2\}$, then 
$
\langle \ch_{k_1} \cdots \ch_{k_N} \rangle_{m, t_1, t_2}'
\in \BD(t_1, t_2)[m]
$
with both degree (in $m$) and weight at most $\sum_{i=1}^N (k_i + 2)$.
\end{enumerate}
\end{theorem}

To prove Theorem~\ref{Intro_BB070340X}, we first analyze 
a different reduced series 
$$
[\ch_{k_1} \cdots \ch_{k_N}]_{m, t_1, t_2}'
$$ 
which is similarly defined as in \eqref{Intro_Reduced.1} and 
\eqref{Intro_trace.0100} but with $\fG_{k_j}(t_1, t_2)$ being replaced 
by $t_2^{k_j} t_1^{\delta(\cdot)} 
\overline{\fB}_{k_j}^{(\alpha)}|_{\alpha = -t_1/t_2}$. 
Applying Theorem~\ref{Intro_TT10210727X}, we draw the conclusions about 
$\langle \ch_{k_1} \cdots \ch_{k_N} \rangle_{m, t_1, t_2}'$ from 
those about $[\ch_{k_1} \cdots \ch_{k_N}]_{m, t_1, t_2}'$.

The second application of the vertex operator feature in 
Theorem~\ref{Intro_TT10210727X} is to determine 
the equivariant higher order derivative of $\fa_{-n}$ with $n > 0$. 

\begin{definition}  \label{Intro_DefEquivaDeri}
Let $\mathfrak f \in \End(\fock')$.
For $k \ge 0$, the (equivariant) derivative $\mathfrak f^{(k)}$ is defined inductively by putting $\mathfrak f^{(0)} = \mathfrak f$ and 
$
\mathfrak f^{(k)} = [\fG_1(t_1, t_2), \mathfrak f^{(k-1)}] 
                      \text{ for $k \ge 1$}.
$
\end{definition}

Applying the commutation relation between the operators 
$\Gamma_+(z)$ and $\Gamma_-(y)$, we derive a closed formula 
(Theorem~\ref{EquivHigherDeriv}) for the derivative $\fa_{-n}^{(k)}$ 
where $n > 0$ and $k \ge 0$. This closed formula is still very complicated, 
but provides an improvement for the non-equivariant setting 
where no closed formula for the derivatives of the Heisenberg operators 
is known (see Theorem~\ref{deriv_th}). 

Finally, the paper is organized as follows. Section~\ref{sect_qMZV} is 
devoted to a quick review of multiple $q$-zeta values. 
Section~\ref{sect_Hilbert} reviews the basics of Hilbert schemes of 
points on surfaces, and the derivatives of the Heisenberg operators 
and the Chern character operators in the non-equivariant setting. 
We prove Theorem~\ref{Intro_TT10210727X} (= Theorem~\ref{TT10210727X}) 
and Theorem~\ref{Intro_BB070340X} (= Theorem~\ref{BB070340X}) in 
Section~\ref{sect_Equivariant} and Section~\ref{sect_Okounkov} respectively. 
In Section~\ref{sect_HigherDeri}, a closed formula for 
the derivative $\fa_{-n}^{(k)}$ with $n > 0$ and $k \ge 0$ is obtained. 

\medskip\noindent
{\bf Convention.} Every cohomology is of $\C$-coefficients unless otherwise specified, and 
$$
(a)_\infty = (a; q)_\infty = \prod_{n=0}^\infty (1-aq^n).
$$


\section{\bf Multiple $q$-zeta values} 
\label{sect_qMZV}

In this section, we recall basic materials related 
to multiple $q$-zeta values.
The following definitions are from \cite{BK1, BK3, Oko}.

\begin{definition}  \label{def_qMZV}
Let $\N = \{1,2,3, \ldots\}$, and fix a subset $S \subset \N$. 
\begin{enumerate}
\item[{\rm (i)}]
Let $Q = \{Q_s(t)\}_{s \in S}$ where each $Q_s(t) \in \Q[t]$ is a polynomial 
with $Q_s(0) = 0$ and $Q_s(1) \ne 0$. For $s_1, \ldots, s_\ell \in S$ 
with $\ell \ge 1$, define
$$
Z_Q(s_1, \ldots, s_\ell)
= \sum_{n_1 > \cdots > n_\ell \ge 1} \prod_{i=1}^\ell 
  \frac{Q_{s_i}(q^{n_i})}{(1-q^{n_i})^{s_i}}
\in \Q[[q]].
$$
Put $Z_Q(\emptyset) = 1$, and define $Z(Q, S)$ to be the $\Q$-linear span 
of the set
$$
\{Z_Q(s_1, \ldots, s_\ell)| \, \ell \ge 0 \text{ and }
s_1, \ldots, s_\ell \in S \}.
$$

\item[{\rm (ii)}]
Define $\mathcal {MD} = Z(Q^E, \N)$ where $Q^E = \{Q_s^E(t)\}_{s \in \N}$
and 
$$
Q_s^E(t) = \frac{tP_{s-1}(t)}{(s-1)!}
$$ 
with $P_s(t)$ being the Eulerian polynomial defined by 
\begin{eqnarray}   \label{def_qMZV.01}
\frac{tP_{s-1}(t)}{(1-t)^s} = \sum_{d=1}^\infty d^{s-1}t^d.
\end{eqnarray}  
We have $t P_{0}(t) = t$.
For $s > 1$, the polynomial $t P_{s-1}(t)$ has degree $s - 1$. 
For $s_1, \ldots, s_\ell \in \N$ with $\ell \ge 1$, define
\begin{eqnarray}   \label{def_qMZV.02}
[s_1, \ldots, s_\ell] 
= Z_{Q^E}(s_1, \ldots, s_\ell)
= \sum_{n_1 > \cdots > n_\ell \ge 1} \prod_{i=1}^\ell 
  \frac{Q_{s_i}^E(q^{n_i})}{(1-q^{n_i})^{s_i}}.
\end{eqnarray}  
This $q$-series is called a {\it bracket} of depth $\ell$ and of weight 
$\sum_{i=1}^\ell s_i$.

\item[{\rm (iii)}]
Define $\text{q}\mathcal {MD}$ be the subspace of $\mathcal {MD}$ linearly 
spanned by $1$ and all the brackets $[s_1, \ldots, s_\ell]$ with $s_1 > 1$. 

\item[{\rm (iv)}]
Define $\qMZV = Z(Q^O, \N_{>1})$ where 
$Q^O = \{Q_s^O(t)\}_{s \in \N_{>1}}$ with 
$$
  Q_s^O(t) 
= \begin{cases}
  t^{s/2},           &\text{if $s \ge 2$ is even,} \\
  t^{(s-1)/2} (t+1), &\text{if $s \ge 3$ is odd,}
  \end{cases}
$$
and $\N_{>1} = \{2,3,4, \ldots\}$.
For $s_1, \ldots, s_\ell \in \N$ with $\ell \ge 1$, define
$$
Z(s_1, \ldots, s_\ell) = Z_{Q^O}(s_1, \ldots, s_\ell).
$$
\end{enumerate}
\end{definition} 

By the Theorem~2.13 and Theorem~2.14 in \cite{BK1},
$\text{q}\mathcal {MD}$ is a subalgebra of $\mathcal {MD}$, 
and $\mathcal {MD}$ is a polynomial ring over $\text{q}\mathcal {MD}$ 
with indeterminate $[1]$:
\begin{eqnarray}   \label{BK1Thm214}
\mathcal {MD} = \text{q}\mathcal {MD}[\,[1]\,].
\end{eqnarray}
By the Proposition~2.2 and Theorem~2.4 in \cite{BK3}, 
$Z(\{Q_s^E(t)\}_{s \in \N_{>1}}, \N_{>1})$ is a subalgebra of 
$\mathcal {MD}$ as well and 
$\qMZV = Z(\{Q_s^E(t)\}_{s \in \N_{>1}}, \N_{>1})$. Therefore,   
\begin{eqnarray}   \label{BK3-2.4}
\qMZV = Z(\{Q_s^E(t)\}_{s \in \N_{>1}}, \N_{>1}) \subset \qMD \subset \MD
\end{eqnarray} 
are inclusions of $\Q$-algebras.  
For instance, by \cite[Example~2.6]{BK3}, 
\begin{eqnarray}   \label{BK3-2.6}
Z(2) = [2], \quad Z(3) = 2[3], \quad Z(4) = [4] - \frac16 [2].
\end{eqnarray}

Next, we recall bi-brackets from \cite{Bac, BK2}.

\begin{definition}   \label{def-BiBracket}
\begin{enumerate}
\item[{\rm (i)}] 
For integers $\ell \ge 0, s_1, \ldots, s_\ell \ge 1$ and 
$r_1, \ldots, r_\ell \ge 0$, define
$$
  \begin{bmatrix} s_1, \ldots, s_\ell\\r_1, \ldots, r_\ell \end{bmatrix}
= \sum_{\substack{u_1 > \cdots > u_\ell > 0\\v_1, \ldots, v_\ell > 0}}
\frac{u_1^{r_1} \cdots u_\ell^{r_\ell}}{r_1! \cdots r_\ell!} \cdot
\frac{v_1^{s_1-1} \cdots v_\ell^{s_\ell-1}}{(s_1-1)! \cdots (s_\ell-1)!} 
\cdot q^{u_1v_1 + \cdots + u_\ell v_\ell} \, \in \, \Q[[q]].
$$
This $q$-series is called a {\it bi-bracket} of depth $\ell$ and of weight 
$\sum_{i=1}^\ell (s_i +r_i)$.

\item[{\rm (ii)}] 
Define $\BD$ to be the vector space spanned by the bi-brackets 
over $\Q$.

\item[{\rm (iii)}] 
Define $\qBD$ to be the subspace spanned by all the
bi-brackets with $s_1 > 1$. 
%
\end{enumerate}
\end{definition}

It is known that 
$$
  \begin{bmatrix} s_1, \ldots, s_\ell\\0, \ldots, 0 \end{bmatrix}
= [s_1, \ldots, s_\ell].
$$
By \cite[Theorem~A~i)]{Bac}, $\BD$ is a $\Q$-algebra and contains 
$\MD$ as a subalgebra. Similar proofs show that 
$\qBD$ is a subalgebra of $\BD$ and contains $\qMD$ as a subalgebra. 
Together with \eqref{BK3-2.4}, 
we obtain inclusions of $\Q$-algebras:
\begin{eqnarray}   \label{Comm_alg}  
\begin{matrix}
&&\qBD &\subset&\BD   \\
&&\cup &&\cup   \\
\qMZV&\subset &\qMD&\subset&\MD.  
\end{matrix}
\end{eqnarray} 

\section{\bf Hilbert schemes of points on surfaces} 
\label{sect_Hilbert}

In this section, we collect background materials 
\cite{Gro, Leh, LQW1, Nak1, Nak2, Qin1} related 
to Hilbert schemes of points on surfaces, 
generalized partitions, Heisenberg algebras of 
Grojnowski and Nakajima, and the derivatives of the Heisenberg operators 
and the Chern character operators in the non-equivariant setting.

Let $X$ be a smooth projective complex surface,
and $\Xn$ be the Hilbert scheme of $n$ points in $X$. 
An element in $\Xn$ is represented by a
length-$n$ $0$-dimensional closed subscheme $\xi$ of $X$. For $\xi
\in \Xn$, let $I_{\xi}$ be the corresponding sheaf of ideals. It
is well-known that $\Xn$ is smooth. 
Define the universal codimension-$2$ subscheme:
\begin{eqnarray*}
{ \mathcal Z}_n=\{(\xi, x) \subset \Xn\times X \, | \, x\in
 {\rm Supp}{(\xi)}\}\subset \Xn\times X.
\end{eqnarray*}
Denote by $p_1$ and $p_2$ the projections of $\Xn \times X$ to
$\Xn$ and $X$ respectively. Let
\begin{eqnarray*}
\fock = \bigoplus_{n=0}^\infty \Hn
\end{eqnarray*}
be the direct sum of total cohomology groups of the Hilbert schemes $\Xn$.
For $m \ge 0$ and $n > 0$, let $Q^{[m,m]} = \emptyset$ and define
$Q^{[m+n,m]}$ to be the closed subset:
$$\{ (\xi, x, \eta) \in X^{[m+n]} \times X \times X^{[m]} \, | \,
\xi \supset \eta \text{ and } \mbox{Supp}(I_\eta/I_\xi) = \{ x \}
\}.$$

We recall Nakajima's definition of the Heisenberg operators \cite{Nak1}. 
Let $\alpha \in H^*(X)$. Set $\mathfrak a_0(\alpha) =0$.
For $n > 0$, the operator $\mathfrak
a_{-n}(\alpha) \in \End(\fock)$ is
defined by
$$
\mathfrak a_{-n}(\alpha)(a) = \w{p}_{1*}([Q^{[m+n,m]}] \cdot
\w{\rho}^*\alpha \cdot \w{p}_2^*a)
$$
for $a \in H^*(X^{[m]})$, where $\w{p}_1, \w{\rho},
\w{p}_2$ are the projections of $X^{[m+n]} \times X \times
X^{[m]}$ to $X^{[m+n]}, X, X^{[m]}$ respectively. Define
$\mathfrak a_{n}(\alpha) \in \End(\fock)$ to be $(-1)^n$ times the
operator obtained from the definition of $\mathfrak
a_{-n}(\alpha)$ by switching the roles of $\w{p}_1$ and $\w{p}_2$. 
We often refer to $\mathfrak a_{-n}(\alpha)$ (resp. $\mathfrak a_n(\alpha)$) 
as the {\em creation} (resp. {\em annihilation})~operator. 
The following is from \cite{Nak1, Gro}. Our convention of the sign follows \cite{LQW2}.

\begin{theorem} \label{commutator}
The operators $\mathfrak a_n(\alpha)$ satisfy
the commutation relation:
\begin{eqnarray*}
\displaystyle{[\mathfrak a_m(\alpha), \mathfrak a_n(\beta)] = -m
\; \delta_{m,-n} \cdot \langle \alpha, \beta \rangle \cdot {\rm Id}_{\fock}}.
\end{eqnarray*}
The space $\fock$ is an irreducible module over the Heisenberg
algebra generated by the operators $\mathfrak a_n(\alpha)$ with a
highest~weight~vector $\vac=1 \in H^0(X^{[0]}) \cong \C$.
\end{theorem}

The Lie bracket in the above theorem is understood in the super
sense according to the parity of the degrees of the
cohomology classes involved. It follows from
Theorem~\ref{commutator} that the space $\fock$ is linearly
spanned by all the Heisenberg monomials $\mathfrak
a_{n_1}(\alpha_1) \cdots \mathfrak a_{n_k}(\alpha_k) \vac$
where $k \ge 0$ and $n_1, \ldots, n_k < 0$.

\begin{definition} \label{VV0614855Z}
\begin{enumerate}
\item[{\rm (i)}]
Let $D_{\lambda}$ denote the Young diagram associated to a partition $\lambda$. 
For a cell $\square$ in $D_{\lambda}$, let $\ell(\square)$, 
$\ell'(\square)$, $a(\square)$ and $a'(\square)$ be respectively
the numbers of cells in $D_{\lambda}$ to the south, north, east and west of 
$\square$ (all excluding $\square$ itself). The numbers $\ell(\square)$ 
and $\ell'(\square)$ are called the {\it leg length} and 
the {\it leg colength} respectively, and $a(\square)$ and $a'(\square)$ 
are called the {\it arm length} and the {\it arm colength} respectively. 
The {\it hook length} $h(\square)$ and the {\it content} $c(\square)$ 
are defined to be $a(\square)+\ell(\square)+1$
and $a'(\square)-\ell'(\square)$ respectively.

\item[{\rm (ii)}] For a partition $\la = (\la_1 \ge \ldots \ge \la_\ell > 0)
= (1^{m_1}2^{m_2}3^{m_3} \cdots)$, define
$$
\mathfrak z_\la = |{\rm Aut}(\la)| \cdot \prod_i \la_i 
= \prod_{i \ge 1} (i^{m_i} m_i!).
$$
\end{enumerate}
\end{definition}

\begin{example}
Let $\la = (5, 5,5,2,1) = (1^12^15^3)$. Its Young diagram $D_\la$
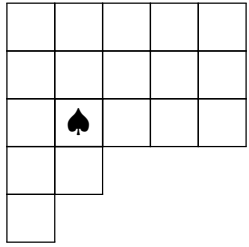
$$
\ytableausetup{centertableaux}
\ytableaushort
{\none,\none,\none \spadesuit \none, \none}
* {5, 5,5,2,1}
$$
contains $18$ cells. For the cell $\spadesuit \in D_\la$, we have 
$\ell(\spadesuit) = 1$, $\ell'(\spadesuit) = 2$, $a(\spadesuit) = 3$, 
$a'(\spadesuit) = 1$,  $h(\spadesuit) = 5$, and $c(\spadesuit) = -1$.
\end{example}

\begin{definition} \label{partition}
\begin{enumerate}
\item[{\rm (i)}]
A {\it generalized partition} \index{partition, generalized} of an integer $n$ is of the form
$$
\lambda = (\cdots  (-2)^{m_{-2}}(-1)^{m_{-1}} 1^{m_1}2^{m_2} \cdots)
$$ 
such that part $i\in \Z$ has multiplicity $m_i$, $m_i \ne 0$ for 
only finitely many $i$'s, and $n = \sum_i i m_i$. Define 
$\ell(\la) = \sum_i m_i$, $|\la| = \sum_i im_i$, 
$\lambda^! = \prod_i m_i!$, 
\begin{eqnarray*}
\lambda^+ &=& (1^{m_1}2^{m_2} \cdots),  \\
\lambda^- &=& (\cdots  (-2)^{m_{-2}}(-1)^{m_{-1}}), \\
-\la &=& (\cdots  (-2)^{m_{2}}(-1)^{m_{1}} 1^{m_{-1}}2^{m_{-2}} \cdots), 
\end{eqnarray*}
$\delta(\la) = \ell(\la^-) - \ell(\la^+)$,
$|\la|_+ = |\la^+| = \sum_{i > 0} im_i$, and $s(\la) = \sum_i i^2m_i$.

\item[{\rm (ii)}] 
Let $\alpha \in H^*(X)$ and $k \ge 1$. Define $\tau_{k*}: H^*(X) \to H^*(X^k)$ 
to be the linear map induced by the diagonal embedding $\tau_k: X \to X^k$, and
$$
(\mathfrak a_{m_1} \cdots \mathfrak a_{m_k})(\alpha)
= \mathfrak a_{m_1} \cdots \mathfrak a_{m_k}(\tau_{*}\alpha)
= \sum_j \mathfrak a_{m_1}(\alpha_{j,1}) \cdots \mathfrak a_{m_k}(\alpha_{j,k})
$$ 
when $\tau_{k*}\alpha = \sum_j \alpha_{j,1} \otimes \cdots 
\otimes \alpha_{j, k}$ via the K\"unneth decomposition of $H^*(X^k)$.

\item[{\rm (iii)}]
For a generalized partition $\lambda = 
\big (\cdots (-2)^{m_{-2}}(-1)^{m_{-1}} 1^{m_1}2^{m_2} \cdots \big )$, define
\begin{eqnarray}
\mathfrak a_{\lambda} 
  &=& \left ( \prod_i (\fa_i)^{m_i} \right ),    \label{partition.1}  \\
\mathfrak a_{\lambda}(\alpha) 
  &=& \left ( \prod_i (\fa_i)^{m_i} \right ) (\alpha)    \label{partition.2}
\end{eqnarray}
where the product $\prod_i (\mathfrak
a_i)^{m_i} $ is understood to be
$\cdots \mathfrak a_{-2}^{m_{-2}} \mathfrak a_{-1}^{m_{-1}}
 \mathfrak a_{1}^{m_{1}} \mathfrak a_{2}^{m_{2}}\cdots$.
%
\end{enumerate}
\end{definition}

For $n > 0$ and a homogeneous class $\alpha \in H^*(X)$, let
$|\alpha| = s$ if $\alpha \in H^s(X)$, and let $G_k(\alpha, n)$ be
the homogeneous component in $H^{|\alpha|+2k}(\Xn)$ of
\begin{eqnarray}    \label{DefOfGGammaN}
 G(\alpha, n) = p_{1*}(\ch({\mathcal O}_{{\mathcal Z}_n}) \cdot p_2^*\alpha
\cdot p_2^*{\rm td}(X) ) \in \Hn
\end{eqnarray}
where $\ch({\mathcal O}_{{\mathcal Z}_n})$ is the Chern
character of ${\mathcal O}_{{\mathcal Z}_n}$
and ${\rm td}(X) $ is the Todd class. We extend the notion $G_k(\alpha,
n)$ linearly to an arbitrary class $\alpha \in H^*(X)$, 
and set $G(\alpha, 0) =0$. 
It was proved in \cite{LQW1} that the cohomology ring of $\Xn$ is
generated by the classes $G_{k}(\alpha, n)$ where $0 \le k < n$
and $\alpha$ runs over a linear basis of $H^*(X)$. 

\begin{definition}   \label{TT12101000Z}
\begin{enumerate}
\item[{\rm (i)}]
For $k \ge 0$ and $\alpha \in H^*(X)$, 
the {\it $k$-th Chern character operator} 
${\mathfrak G}_k(\alpha) \in \End({\fock})$ is the operator acting 
on the $n$-th component $H^*(\Xn)$ by the cup product 
with $G_k(\alpha, n)$. 

\item[{\rm (ii)}]
Let $1_X$ be the fundamental class of the projective surface $X$. 
For $k \ge 0$ and $\mathfrak f \in \End(\mathbb H_X)$, 
define the {\it $k$-th derivative} 
$\mathfrak f^{(k)}$ of $\mathfrak f$ inductively by 
$\mathfrak f^{(0)} = \mathfrak f$ and
$
\mathfrak f^{(k)} = [\fG_1(1_X), \mathfrak f^{(k-1)}]
$ 
for $k \ge 1$. 
\end{enumerate}
\end{definition}

The following two theorems are from \cite{LQW2}.

\begin{theorem} \label{deriv_th}
Let $k \ge 0$, $n \in \Z$, and $\alpha\in H^*(X)$. Then,
$\mathfrak a_n^{(k)}(\alpha)$ is equal to
\begin{eqnarray*}
& &(-n)^k k! \left ( \sum_{\ell(\lambda) = k+1, |\lambda|=n}
   \frac{\fa_{\la}}{\lambda^!}(\alpha) -
   \sum_{\ell(\lambda) = k-1, |\lambda|=n}
   \frac{\lambsq - 1}{24}
   \frac{\fa_{\la}}{\lambda^!}(e_X \alpha) \right )   \\
&+&\sum_{\ell(\lambda) = k, |\lambda|=n}
   f_{1, \lambda} \frac{\fa_{\la}}{\lambda^!}(K_X \alpha)
   + \sum_{\ell(\lambda) = k-1, |\lambda|=n}
   f_{2, \lambda} \frac{\fa_{\la}}{\lambda^!}(K_X^2 \alpha)
\end{eqnarray*}
where all the numbers $f_{1, \lambda}$ and $f_{2, \lambda}$ are independent of 
$X$ and $\alpha$, and $K_X$ and $e_X$ are the canonical class and Euler class 
of $X$ respectively.
\end{theorem}

\begin{theorem} \label{char_th}
Let $k \ge 0$ and $\alpha, \beta \in H^*(X)$. 
Then, $\mathfrak G_k(\alpha)$ is equal to
\begin{eqnarray*}
& &- \sum_{\ell(\lambda) = k+2, |\lambda|=0}
   \frac{\fa_{\la}}{\lambda^!}(\alpha)
   + \sum_{\ell(\lambda) = k, |\lambda|=0}
   \frac{\lambsq - 2}{24}
   \frac{\fa_{\la}}{\lambda^!}(e_X \alpha)  \\
&+&\sum_{\ell(\lambda) = k+1, |\lambda|=0} g_{1, \lambda}
   \frac{\fa_{\la}}{\lambda^!}(K_X \alpha)
   + \sum_{\ell(\lambda) = k, |\lambda|=0} g_{2, \lambda}
   \frac{\fa_{\la}}{\lambda^!}(K_X^2 \alpha)
\end{eqnarray*}
where all the numbers $g_{1, \lambda}$ and $g_{2, \lambda}$ are 
independent of $X$ and $\alpha$. 
\end{theorem}

Via direct computations, Tang \cite[Proposition 5.6]{Tang} proves that 
$\mathfrak G_2(\alpha)$ equals
$$
- \sum_{\ell(\lambda) = 4, |\lambda|=0}
   \frac{\mathfrak a_{\lambda}(\alpha)}{\lambda^!} 
   + \sum_{n > 0} \frac{n^2 - 1}{12}
   \fa_{-n}\fa_n(e_X \alpha)      
$$
\begin{eqnarray}   \label{Tang5.6}
- \sum_{\ell(\lambda) = 3, |\lambda|=0} \frac{|\la^+|-1}{2} 
   \frac{\mathfrak a_{\lambda}(K_X \alpha)}{\la^!}
   - \sum_{n > 0} \frac{(n-1)(2n-1)}{12} \, \fa_{-n}\fa_n(K_X^2 \alpha)
\end{eqnarray}
where $\la^+$ denotes the positive part of a generalized partition 
$\la$. 

To determine the universal numbers $f_{1, \la}, f_{2, \la}, 
g_{1, \la}, g_{2, \la}$ in Theorem~\ref{deriv_th} and Theorem~\ref{char_th},
it suffices to determine the operators $\mathfrak a_n^{(k)}(1_X)$ 
with $n < 0$ and $\fG_k(1_X)$ for smooth projective toric surfaces $X$. 
Since the affine plane $\C^2$ is the basic model for such surfaces,
we will study the equivariant counterparts of these operators 
for the Hilbert schemes $(\C^2)^{[n]}$ in the rest of this paper.

\section{\bf Equivariant Chern character operators for $(\C^2)^{[n]}$} 
\label{sect_Equivariant}

Throughout the rest of the paper, we fix the action 
\eqref{Intro_TT08250207X} of $\T = (\C^*)^2$ on $X = \C^2$.
In this section, we recall the equivariant Heisenberg operators which 
act on the equivariant cohomology of the Hilbert schemes $(\C^2)^{[n]}$, 
and review the connection between this setting and 
the ring of symmetric functions. Then we prove 
Theorem~\ref{Intro_TT10210727X} (= Theorem~\ref{TT10210727X}).

\subsection{\bf Equivariant Heisenberg operators for $(\C^2)^{[n]}$} 
\label{subsect_EquivHei}
$\,$
\medskip

Let $X = \C^2$. Recall from \eqref{Intro_def_HX'} the space $\fock'$ of 
the equivariant cohomology of the Hilbert schemes $\Xn$. 
For the Heisenberg algebra action on $\fock'$, we follow the presentations 
in \cite[Section~2]{OP1} and \cite[Section~2]{OP2}.
By definition, the Fock space $\mathfrak F$ is freely generated over $\C$ 
by commuting creation operators $\fa_{-k}, k \in \Z_{>0}$ 
acting on the vacuum vector $\vac$. The annihilation operators 
$\fa_{k}, k \in \Z_{>0}$ satisfy $\fa_{k} \vac = 0$, 
and obey the commutation relations \eqref{Intro_TT09110234X}.
The Fock space $\mathfrak F$ has a linear basis consisting of the elements 
\begin{eqnarray}  \label{TT09110240X}
  |\la \rangle 
= \frac{1}{\mathfrak z_\la} \fa_{-\la} \vac
= \frac{1}{\mathfrak z_\la} \prod_i \fa_{-\la_i} \cdot \vac
\end{eqnarray}
where $\la$ denotes a (usual) partition with parts $\la_i > 0$,
and the notation $\fa_{-\la}$ is from \eqref{partition.1}. 
Then there is a linear isomorphism 
\begin{eqnarray}  \label{TT09110237X}
\Psi: \quad \mathfrak F \otimes_\C \C(t_1, t_2) \to \fock'
\end{eqnarray}
sending the element $\prod_i \la_i \cdot |\la \rangle$ to 
the equivariant cohomology class associated to 
the $\T$-equivariant closed subvariety 
of $X^{[|\la|]}$ whose generic elements are given by a union of 
closed subschemes of lengths $\la_1, \ldots, \la_{\ell(\la)}$
supported at $\ell(\la)$ distinct points of $X$. 
The vacuum vector $\vac$ corresponds to the unit in 
$H^*_\T(X^{[0]})' = \C(t_1, t_2)$, and 
\begin{eqnarray}  \label{BB02051005Z}
\Psi(|\la \rangle), \quad \la \vdash n
\end{eqnarray}
form a linear basis of $H_\T^*(\Xn)^{\prime}$.

Alternatively, the Heisenberg operators 
$\fp_m(\alpha^\T) \in {\rm End}(\fock')$ 
with $m \in \Z, \alpha^\T \in H^*_\T(X)'$ are defined in \cite{LQW4} 
(see also \cite{Vas} and \cite[Sections~11.2]{Qin1}). 
These Heisenberg operators satisfy the relation
\begin{eqnarray}   \label{Heisenberg}
[\mathfrak p_k(\alpha^\T), \mathfrak p_{\ell}(\beta^\T)]
= k\delta_{k, -\ell} \cdot \langle \alpha^\T, \beta^\T \rangle \cdot \text{\rm Id}
\end{eqnarray}
for $\alpha^\T, \beta^\T \in H^*_\T(X)'$. In \eqref{Heisenberg}, 
$\langle \cdot, \cdot \rangle$ is a nondegenerate bilinear form on $H^*_\T(X)'$. 
For $k > 0$, we see from \cite[Lemma~14.1]{Qin1} that
\begin{eqnarray}   \label{TT09120954Z}
\fa_{-k} = \fp_{-k}(1_X), \quad \fa_{k} = -\fp_{k}([x]) 
\end{eqnarray}
where $1_X = [X]$ and $[x]$ denote the equivariant fundamental classes 
of $X$ and the origin $x \in X = \C^2$ respectively. 
Via the linear isomorphism \eqref{TT09110237X}, 
\eqref{TT09120954Z} enables us to interchange the operators $\fa_n, n \in \Z$
with the operators $\fp_{-k}(1_X), \fp_{k}([x])$ with $k > 0$, or vice versa. 

Let $\Lambda_\C$ be the ring (over $\C$) of symmetric functions in 
countably many 
independent variables $x_1, x_2, \ldots$. Let $p_\la \in \Lambda_\C$ 
be the power symmetric function associated to a partition $\la$,
and $J_\la^{(\alpha)} \in \Lambda_\C \otimes_\C \C(\alpha)$ be 
the integral form of the Jack symmetric function 
depending on the parameter $\alpha$ \cite[(10.22) on p.381]{Mac}. 
The linear map 
\begin{eqnarray}  \label{PLaALa}
p_\la \mapsto \mathfrak z_\la \cdot |\la \rangle = \fa_{-\la} \vac
\end{eqnarray}
identifies $\Lambda_\C$ with the Fock space $\mathfrak F$. Let 
\begin{eqnarray}  \label{TT09120958Z}
{\bf J}^\la(t_1, t_2) \in \mathfrak F \otimes_\C \C[t_1, t_2]
\subset \mathfrak F \otimes_\C \C(t_1, t_2)
\end{eqnarray}
be the inverse image of the equivariant class 
$[\xi_\la] = i_{\lambda!}(1_{\xi_{\lambda}}) \in \fock'$
where $\xi_{\lambda}$ is the fixed point defined by \eqref{Intro_VV06141023Z} 
and $i_\la: \xi_{\lambda} \to X$ is the inclusion map. Then 
\begin{eqnarray}  \label{JLaT1T2}
{\bf J}^\la(t_1, t_2) 
= t_2^{|\la|} t_1^{\ell(\cdot)} J_\la^{(\alpha)}|_{\alpha = -t_1/t_2}
\end{eqnarray}
where $t_1^{\ell(\cdot)}$ acts on the basis element $|\mu \rangle$ 
by $t_1^{\ell(\cdot)}|\mu \rangle = t_1^{\ell(\mu)}|\mu \rangle$.
By \cite[(10.24) on p.381]{Mac},
$$
{\bf J}^\la(t_1, t_2) = {\bf J}^{\la'}(t_2, t_1)
$$
where $\la'$ denotes the conjugate of the partition $\la$.

Recall the operators $\Gamma_\pm(z)$ from \eqref{Intro_GammaLz}.
Their basic and well-known properties are proved here for completeness. 

\begin{lemma}  \label{GammaPMComm}
\begin{enumerate}
\item[{\rm (i)}]
Let $n > 0$. Then, we have
\begin{eqnarray}   
[\Gamma_+(z)^r, \fa_{-n}] &=& r z^{-n} \Gamma_+(z)^r, 
          \label{GammaPMComm.01}    \\
{[\Gamma_-(z)^r, \fa_{n}]} &=& -r z^{n} \Gamma_-(z)^r.
          \label{GammaPMComm.02}
\end{eqnarray}

\item[{\rm (ii)}]
$\Gamma_+(z)^a \Gamma_-(y)^b 
= (1 - z^{-1}y)^{-ab} \cdot \Gamma_-(y)^b \Gamma_+(z)^a$.
\end{enumerate}
\end{lemma}
\begin{proof}
(i) It suffices to prove \eqref{GammaPMComm.01}. 
By \eqref{Intro_GammaLz} and \eqref{Intro_TT09110234X},
\begin{eqnarray*}    
   [\Gamma_+(z)^r, \fa_{-n}]
&=&\left [\exp \left ( r \sum_{k > 0} \frac{z^{-k}}{k} \fa_{k} \right ),
   \fa_{-n} \right ]     \\
&=&\exp \left ( r \sum_{k > 0, k \ne n} \frac{z^{-k}}{k} \fa_{k} \right )
   \left [\exp \left ( r \frac{z^{-n}}{n} \fa_{n} \right ),
   \fa_{-n} \right ]     \\
&=&r z^{-n} \exp \left ( r \sum_{k > 0} \frac{z^{-k}}{k} \fa_{k} \right )  \\
&=&r z^{-n} \Gamma_+(z)^r.
\end{eqnarray*}


(ii) As in the preceding paragraph, a straight-forward calculation shows that
\begin{eqnarray*}  
& &\exp \left ( \frac{z^{-n}}{n} a\fa_{n} \right ) 
   \exp \left ( \frac{y^{n}}{n} b\fa_{-n} \right )  \\
&=&\exp \left ( \frac{(z^{-1}y)^n}{n} ab\right ) 
  \cdot \exp \left ( \frac{y^{n}}{n} b\fa_{-n} \right )
  \exp \left ( \frac{z^{-n}}{n} a\fa_{n} \right ).
\end{eqnarray*}
It follows that $\Gamma_+(z)^a \Gamma_-(y)^b 
= (1 - z^{-1}y)^{-ab} \cdot \Gamma_-(y)^b \Gamma_+(z)^a$.
\end{proof}
\subsection{\bf Equivariant Chern character operators} 
\label{subsect_EquivChern}
$\,$
\medskip

Recall the equivariant Chern character operators $\fG_k(t_1, t_2)$ from 
Definition~\ref{Intro_DefFG}. By the localization theorem, 
the equivariant classes $[\xi_\la], \la \vdash n$ form a linear basis of 
$H_\T^*(\Xn)^{\prime}$. It follows from \eqref{Intro_DefFG.2} that 
$\fG_k(t_1, t_2)$ preserves the component $H_\T^*(\Xn)^{\prime}$ 
in $\fock'$. Therefore, $\fG_k(t_1, t_2)$ is of the form
\begin{eqnarray}    \label{BB02061023Z}
\fG_k(t_1, t_2) = \sum_{|\la| = 0} g^{(k)}_\la \frac{\fa_\la}{\la^!}
\end{eqnarray} 
where $g^{(k)}_\la \in \C(t_1, t_2)$.
Since ${\bf J}^\la(t_1, t_2)$ corresponds to $[\xi_\la]$, 
we rewrite \eqref{Intro_DefFG.2} as
\begin{eqnarray}  
  \fG_k(t_1, t_2) {\bf J}^\la(t_1, t_2)
= \sum_{\square\in D_\lambda} \frac{(-1)^k}{k!} 
  (a'(\square)t_1 + \ell'(\square)t_2)^k \cdot {\bf J}^\la(t_1, t_2).
  \label{TT09161027Z.2}  
\end{eqnarray}

Next, we recall the transformed Macdonald symmetric functions 
$\W H_\la(X; q, t)$ from \cite[(11)]{GH} 
(see also \cite[Definition 3.5.2]{Hai}), 
where $q$ and $t$ are two formal variables. 
For a (usual) partition $\la = (\la_1 \ge \ldots \ge \la_\ell > 0)$, put
$$
n(\la) = \sum_{i=1}^\ell (i - 1) \la_i
= \sum_{\square\in D_\lambda} \ell(\square).
$$
In terms of the plethystic or $\la$-ring notation, 
$\W H_\la(X; q, t)$ is defined by
\begin{eqnarray}  \label{DefWHla}
  \W H_\la(X; q, t)
&=&t^{n(\la)} \cdot
   J_\la\left [\frac{X}{1-t^{-1}}; q, t^{-1} \right ]  \nonumber   \\
&=&t^{n(\la)} \cdot \prod_{\square \in D_{\lambda}} 
   (1 - q^{a(\square)}t^{-\ell(\square)-1}) \cdot
   P_\la\left [\frac{X}{1-t^{-1}}; q, t^{-1}\right ] \quad
\end{eqnarray}
where $X = x_1 + x_2 + \cdots$ is the sum of the variables, 
and $J_\la(X; q, t)$ is the integral form of the Macdonald symmetric 
function $P_\la(X; q, t)$ \cite[Section~VI.8]{Mac}.
The functions $\big \{\W H_\la(X; q, t)\big \}_\la$ form a basis of 
$\Lambda_\C \otimes_\C \C(q, t)$. Equivalently,
\begin{eqnarray}  \label{DefWHla.1}
  J_\la(X; q, t)
= t^{n(\la)} \cdot \W H_\la\left [(1-t)X; q, t^{-1} \right ]. 
\end{eqnarray}
Also recall from \cite[(10.23) on p.381]{Mac} that
\begin{eqnarray}  \label{JAlphaLaLimit}
  J_\la^{(\alpha)} 
= \lim_{t \to 1} (1-t)^{-|\la|} J_\la(X; t^\alpha, t).
\end{eqnarray}

The following deformed vertex operator is introduced in \cite{CW}:
\begin{eqnarray}  \label{CW1}
   V(z; q, t, \w q, \w t) 
&=&\exp\left (\sum_{k \ge 1} (q^k - \w q^k) \frac{z^k}{k}\fa_{-k} \right )
   \exp\left (\sum_{k \ge 1} (\w t^k - t^k) \frac{z^{-k}}{k}\fa_k \right ) 
   \nonumber   \\
&=&\Gamma_-(qz) \Gamma_-(\w qz)^{-1} 
   \Gamma_+(\w t^{-1}z) \Gamma_+(t^{-1}z)^{-1}.
\end{eqnarray}
Inspired by this definition, we define the vertex operator
\begin{eqnarray}  \label{DefWVy1y2}
  \W V(z; q, t, \w q, \w t)
= \exp\left (-\sum_{k \ge 1} \frac{z^k}{k}\fa_{-k} \right )
  \exp\left (\sum_{k \ge 1} (\w q^k - q^k)(\w t^k - t^k) 
  \frac{z^{-k}}{k}\fa_k \right ).
\end{eqnarray}
Expand $V(z; q, t, \w q, \w t)$ and $\W V(z; q, t, \w q, \w t)$
in terms of the powers of $z$:
\begin{eqnarray}  
   V(z; q, t, \w q, \w t)
&=&\sum_{m \in \Z} V_m(z; q, t, \w q, \w t) z^m,  \label{CW2.1}  \\
   \W V(z; q, t, \w q, \w t)
&=&\sum_{m \in \Z} \W V_m(z; q, t, \w q, \w t) z^m.  \label{CW2.2}
\end{eqnarray}

Define the linear operator $\fB(q, t)$ on 
$\Lambda_\C \otimes_\C \C(q, t)$ by putting
\begin{eqnarray}  \label{fBt1t2}
  \fB(q, t) \W H_\la(X; q, t) 
= \sum_{\square\in D_\lambda} 
  q^{a'(\square)}t^{\ell'(\square)} \cdot \W H_\la(X; q, t)
\end{eqnarray}
for every partition $\la$. By \cite[Theorem 3.2]{GH} 
or \cite[Corollary~3.5.5]{Hai},  
\begin{eqnarray}  \label{CWRemark12}
  \fB(q, t) 
= \frac{1}{(1-q)(1-t)} \big (1 - \W V_0(z; q, t, 1, 1)\big ).
\end{eqnarray}

\begin{lemma}  \label{TT10180917Z}
The symmetric function $J_\la(X; q, t)$ is an eigenfunction with 
eigenvalue $\sum_{\square\in D_\lambda} q^{a'(\square)}t^{-\ell'(\square)}$ 
of the operator
\begin{eqnarray}  \label{TT10180917Z.0}
  \overline{\fB}(q, t^{-1})
= \frac{1}{(1-q)(1-t^{-1})} \big (1 - V_0(z; t, t^{-1}, 1, qt^{-1})\big ).
\end{eqnarray}
\end{lemma}
\begin{proof}
By \eqref{DefWHla.1}, we have $J_\la(X; q, t) = t^{n(\la)} \cdot
\W H_\la\left [(1-t)X; q, t^{-1} \right ]$.
Regard the function $\W H_\la(X; q, t^{-1})$ as a polynomial of 
the power symmetric functions $p_1, p_2, \cdots$. 
Then, $\W H_\la\left [(1-t)X; q, t^{-1} \right ]$ is obtained 
from $\W H_\la(X; q, t^{-1})$ via the substitution $p_k \mapsto 
(1 - t^{k})p_k$ for every $k \ge 1$. Therefore, by \eqref{fBt1t2},
\begin{eqnarray}  \label{TT10180917Z.1}
  \overline{\fB}(q, t^{-1}) J_\la(X; q, t) 
= \sum_{\square\in D_\lambda} 
  q^{a'(\square)}t^{-\ell'(\square)} \cdot J_\la(X; q, t)
\end{eqnarray}
where $\overline{\fB}(q, t^{-1})$ is obtained from $\fB(q, t^{-1})$
via the substitution $p_k \mapsto (1 - t^{k})p_k$ for every $k \ge 1$. 
By \eqref{CWRemark12}, we obtain
\begin{eqnarray}  \label{TT10180917Z.2}
  \fB(q, t^{-1}) 
= \frac{1}{(1-q)(1-t^{-1})} \big (1 - \W V_0(z; q, t^{-1}, 1, 1)\big ).
\end{eqnarray}
For $k > 0$, via the identificaition \eqref{PLaALa}, the operators $\fa_{-k}$ 
and $\fa_k$ on the space $\mathfrak F$ are identified with $p_k \cdot $ 
and $k \cdot \partial/\partial p_k$ on $\Lambda_\C$ respectively.
So regarded as an operation on $\Lambda_\C \otimes_\C \C(q, t)$,
we see from \eqref{DefWVy1y2} that $\W V(z; q, t^{-1}, 1, 1)$ is equal to 
$$
\exp\left (-\sum_{k \ge 1} \frac{z^k}{k}p_{k} \right )
  \exp\left (\sum_{k \ge 1} (1 - q^k)(1 - t^{-k}) 
  z^{-k} \frac{\partial}{\partial p_k} \right ).
$$
Under the substitution $p_k \mapsto (1 - t^{k})p_k$ for every $k \ge 1$,
$\W V(z; q, t^{-1}, 1, 1)$ becomes
$$
\exp\left (\sum_{k \ge 1} (t^{k} - 1)\frac{z^k}{k}p_{k} \right )
  \exp\left (\sum_{k \ge 1} (1 - q^k)(-t^{-k})
  z^{-k} \frac{\partial}{\partial p_k} \right )
$$
which is $V(z; t, t^{-1}, 1, qt^{-1})$ via 
the identificaition \eqref{PLaALa}. Hence by \eqref{TT10180917Z.2},
$$
\overline{\fB}(q, t^{-1})
= \frac{1}{(1-q)(1-t^{-1})} \big (1 - V_0(z; t, t^{-1}, 1, qt^{-1})\big ).
$$
Now our lemma follows immediately from \eqref{TT10180917Z.1}.
\end{proof}

When $q = e^{t_1}$ and $t = e^{t_2}$, we write $\overline{\fB}(q, t^{-1})$ as 
$\overline{\fB}(t_1, -t_2)$ and regard $\overline{\fB}(q, t^{-1})
= \overline{\fB}(t_1, -t_2)$ as an operator on $\fock'$ via the identifications
$$
(\Lambda_\C \otimes_\C \C(q, t)) \otimes_{\C(q, t)} \C(t_1, t_2) 
\cong \Lambda_\C \otimes_\C \C(t_1, t_2)
\cong \fock'.
$$

\begin{lemma}  \label{LMAfBy1y2}
Let $q = e^{t_1}$ and $t = e^{t_2}$. Then, 
$\overline{\fB}(q, t^{-1}) = \overline{\fB}(t_1, -t_2)$ is equal to
$$
\sum_{|\la| = 0} \frac{\fa_\la}{\la^!} 
t_1^{\ell(\la^+) - 1}t_2^{\ell(\la^-) - 1} \left ( 
1 + \frac{|\la^+|-1}{2} (t_1-t_2) \right.
$$
\begin{eqnarray}  \label{LMAfBy1y2.0}
\left.
+ \left (\frac{1}{12} - \frac{|\la^+|}{4} + \frac{|\la^+|^2}{8}\right )(t_1-t_2)^2
+ \frac{s(\la^+)}{24}t_1^2 - \frac{1}{12}t_1t_2
+ \frac{s(\la^-)}{24}t_2^2 + O(3) \right )
\end{eqnarray}
where $\la^-$ and $\la^+$ denote
the negative part and positive part of $\la$ respectively,
and $O(3)$ denotes a formal power series in $t_1$ and $t_2$ with 
lowest degree $\ge 3$.
\end{lemma}
\begin{proof}
By \eqref{TT10180917Z.0}, \eqref{CW2.1} and \eqref{CW1},
$\overline{\fB}(q, t^{-1})$ is equal to
$$
\frac{\Coe_{z^0}}{(1-q)(1-t^{-1})} \left (1 - 
   \exp\left (\sum_{k \ge 1} (t^{k} - 1) \frac{z^k}{k}\fa_{-k} \right )
   \exp\left (\sum_{k \ge 1} (q^k - 1)t^{-k} 
   \frac{z^{-k}}{k}\fa_k \right )\right ).
$$
Expanding the two exponential functions, we see that 
$$
\overline{\fB}(q, t^{-1})
= \sum_{|\la| = 0} \frac{\fa_\la}{\la^!} \frac{-1}{(q-1)(t^{-1}-1)}
   \prod_{k > 0} \left (\frac{t^{k} - 1}{k}\right )^{m_{-k}} 
   \cdot \prod_{k > 0} \left (\frac{(q^k - 1)t^{-k}}{k}\right )^{m_k}
$$
where $m_i$ with $i \ne 0$ denotes the multiplicity of part $i$ in 
the generalized partition $\la$. Since $q = e^{t_1}$ and $t = e^{t_2}$,
we conclude that $\overline{\fB}(q, t^{-1})$ is equal to
$$
\sum_{|\la| = 0} \frac{\fa_\la}{\la^!} 
\frac{t_1^{\ell(\la^+) - 1}t_2^{\ell(\la^-) - 1}}
     {\sum_{i \ge 1} t_1^{i-1}/i! \cdot \sum_{i \ge 1} (-t_2)^{i-1}/i!} 
$$
\begin{eqnarray}  \label{LMAfBy1y2.1}
\cdot \prod_{k > 0} \left (\sum_{i \ge 1} 
  \frac{(kt_2)^{i-1}}{i!} \right )^{m_{-k}} \cdot 
  \prod_{k > 0} \left (\sum_{i \ge 1} \frac{(kt_1)^{i-1}}{i!} 
  \cdot \sum_{i \ge 0} \frac{(-kt_2)^i}{i!} \right )^{m_k}.
\end{eqnarray}

Note that $1/\big ({\sum_{i \ge 1} t_1^{i-1}/i! 
\cdot \sum_{i \ge 1} (-t_2)^{i-1}/i!}\big )$ is equal to
\begin{eqnarray}  \label{LMAfBy1y2.2}
1 - \frac12(t_1-t_2) + \frac{1}{12}\big ( (t_1-t_2)^2 - t_1t_2\big ) + O(3).
\end{eqnarray}
Furthermore, $\left (\sum_{i \ge 1} {(kt_2)^{i-1}}/{i!} \right )^{m_{-k}}$
is equal to
\begin{eqnarray*}   
& &1 + m_{-k} \left (\sum_{i \ge 2} \frac{(kt_2)^{i-1}}{i!} \right ) 
   + \binom{m_{-k}}{2} \left (\sum_{i \ge 2} \frac{(kt_2)^{i-1}}{i!} \right )^2 
   + O(3)   \\
&=&1 + \frac{k m_{-k}}{2} t_2 + \left (\frac{k^2 m_{-k}^2}{8} 
   + \frac{k^2 m_{-k}}{24} \right ) t_2^2 + O(3).
\end{eqnarray*}
So we see that $\prod_{k > 0} \left (\sum_{i \ge 1} {(kt_2)^{i-1}}/{i!} 
\right )^{m_{-k}}$ is equal to
\begin{eqnarray}  \label{LMAfBy1y2.3}
1 + \frac{|\la^+|}{2} t_2 + \left (\frac{|\la^+|^2}{8} 
   + \frac{s(\la^-)}{24} \right ) t_2^2 + O(3)
\end{eqnarray}
noting that $\sum_{k > 0} k m_{-k} = \sum_{k > 0} k m_{k} = |\la^+|$. 
Similarly,
\begin{eqnarray}  \label{LMAfBy1y2.4}
& &\prod_{k > 0} \left (\sum_{i \ge 1} \frac{(kt_1)^{i-1}}{i!} 
  \cdot \sum_{i \ge 0} \frac{(-kt_2)^i}{i!} \right )^{m_k}  \nonumber  \\
&=&1 + |\la^+| \left (\frac12 t_1 - t_2\right )
   + \frac{s(\la^+)}{24} t_1^2 
   + \frac{|\la^+|^2}{2}\left (\frac12 t_1 - t_2\right )^2 + O(3).
\end{eqnarray}
Combining \eqref{LMAfBy1y2.1}, \eqref{LMAfBy1y2.2}, \eqref{LMAfBy1y2.3}
and \eqref{LMAfBy1y2.4} proves the lemma.
\end{proof}

\begin{lemma}    \label{TT10210857Z}
Let $q = e^{\alpha t_0}$ where $\alpha$ is a formal parameter, and $t = e^{t_0}$. 
For $k \ge 0$, let $\overline{\fB}_k^{(\alpha)} 
= \Coe_{t_0^k} \overline{\fB}(q, t^{-1})$. 
If $0 \le k \le 2$, then 
\begin{eqnarray}    \label{TT10210857Z.0}
  \overline{\fB}_k^{(\alpha)} J_\la^{(\alpha)} 
= \sum_{\square\in D_\lambda} \frac{(a'(\square) \alpha - \ell'(\square))^k}{k!} 
  \cdot J_\la^{(\alpha)}, \qquad \forall \la.
\end{eqnarray}
\end{lemma}
\begin{proof}
Setting $t_1 = \alpha t_0$ and $t_2 = t_0$ in Lemma~\ref{LMAfBy1y2}, we see that
$\overline{\fB}(q, t^{-1})$ is equal to
$$
\sum_{|\la| = 0} \frac{\fa_\la}{\la^!} 
\alpha^{\ell(\la^+) - 1} t_0^{\ell(\la) - 2} \left ( 
1 + \frac{|\la^+|-1}{2} (\alpha - 1)t_0 \right.
$$
\begin{eqnarray}   \label{TT10210857Z.1}
\left.
+ \left (\frac{1}{12} - \frac{|\la^+|}{4} + \frac{|\la^+|^2}{8}\right )
(\alpha-1)^2 t_0^2 + \frac{s(\la^+)\alpha^2 - 2\alpha + s(\la^-)}{24} t_0^2 
+ O(3) \right ).
\end{eqnarray}
It follows that $\overline{\fB}_0^{(\alpha)} = \sum_{i > 0} \fa_{-i}\fa_{i}$,
and \eqref{TT10210857Z.0} holds when $k = 0$. Also,
\begin{eqnarray}   
   \overline{\fB}_1^{(\alpha)} 
&=&\sum_{\ell(\la) = 3, |\la| = 0} \alpha^{\ell(\la^+) - 1} \frac{\fa_\la}{\la^!}
    + \sum_{i > 0} \frac{i-1}{2} (\alpha - 1) \fa_{-i}\fa_{i},
       \label{TT10210857Z.2}    \\
   \overline{\fB}_2^{(\alpha)}
&=&\sum_{\ell(\la) = 4, |\la| = 0} \alpha^{\ell(\la^+) - 1} \frac{\fa_\la}{\la^!} 
    + \sum_{\ell(\la) = 3, |\la| = 0} \alpha^{\ell(\la^+) - 1}
    \frac{|\la^+|-1}{2} (\alpha - 1) \frac{\fa_\la}{\la^!}    \nonumber \\
& &+ \sum_{i > 0} \left (
    \left (\frac{1}{12} - \frac{i}{4} + \frac{i^2}{8}\right )(\alpha-1)^2 
    + \frac{i^2\alpha^2 - 2\alpha + i^2}{24} \right ) \fa_{-i}\fa_{i}.  
       \label{TT10210857Z.3}
\end{eqnarray}

By \eqref{JAlphaLaLimit}, $J_\la^{(\alpha)} 
= \displaystyle{\lim_{t \to 1}} (1-t)^{-|\la|} J_\la(X; t^\alpha, t)
= \displaystyle{\lim_{t_0 \to 0}} (1-e^{t_0})^{-|\la|} 
  J_\la(X; e^{\alpha t_0}, e^{t_0})$.
Let $\la'$ denote the conjugate of the partition $\la$. 
Then by \cite[(8.5) on p.353]{Mac}, 
$$
  J_\la(X; q^{-1}, t^{-1}) 
= (-1)^{|\la|} q^{-n(\la')} t^{-n(\la) - |\la|} J_\la(X; q, t).
$$
Thus, $(1 - t^{-1})^{-|\la|} J_\la(X; q^{-1}, t^{-1})
= q^{-n(\la')} t^{-n(\la)} \cdot (1-t)^{-|\la|} J_\la(X; q, t)$. 
So
\begin{eqnarray}   \label{TT10210857Z.4}
  \frac{J_\la(X; q, t)}{(1-t)^{|\la|}}
= J_\la^{(\alpha)} 
  + \frac{n(\la') \alpha + n(\la)}{2} J_\la^{(\alpha)} \cdot t_0 + O(2).
\end{eqnarray}

Next, $\overline{\fB}(q, t^{-1}) 
= \sum_{k \ge 0} \overline{\fB}_k^{(\alpha)} t_0^k$ 
in terms of our notation. By Lemma~\ref{TT10180917Z},
\begin{eqnarray}    \label{TT10210857Z.5}
  \sum_{k \ge 0} \overline{\fB}_k^{(\alpha)} t_0^k 
  \frac{J_\la(X; q, t)}{(1-t)^{|\la|}}
= \sum_{k \ge 0} c_{\la, k}^{(\alpha)} t_0^k 
  \cdot \frac{J_\la(X; q, t)}{(1-t)^{|\la|}}
\end{eqnarray}
where $c_{\la, k}^{(\alpha)} = \sum_{\square \in D_\lambda} 
{(a'(\square) \alpha - \ell'(\square))^k}/{k!}$. Since 
$$
\overline{\fB}_0^{(\alpha)} J_\la(X; q, t) = |\la| \cdot J_\la(X; q, t),
$$
we can drop the index $k = 0$ 
from both sides of \eqref{TT10210857Z.5}. It follows that 
\begin{eqnarray}    \label{TT10210857Z.6}
   \sum_{k \ge 1} \overline{\fB}_k^{(\alpha)} t_0^{k-1} 
   \frac{J_\la(X; q, t)}{(1-t)^{|\la|}}   
=  \sum_{k \ge 1} c_{\la, k}^{(\alpha)} t_0^{k-1} 
   \cdot \frac{J_\la(X; q, t)}{(1-t)^{|\la|}}.   
\end{eqnarray}
Letting $t_0 \to 0$ yields $\overline{\fB}_1^{(\alpha)} J_\la^{(\alpha)} 
= c_{\la, 1}^{(\alpha)} \cdot J_\la^{(\alpha)}$. So
\eqref{TT10210857Z.0} holds for $k = 1$.

To prove that \eqref{TT10210857Z.0} holds for $k = 2$, 
we subtract $\overline{\fB}_1^{(\alpha)} J_\la^{(\alpha)} 
= c_{\la, 1}^{(\alpha)} \cdot J_\la^{(\alpha)}$ from \eqref{TT10210857Z.6}, 
followed by dividing by $t_0$ on both sides. We obtain
\begin{eqnarray*}
& &\sum_{k \ge 2} \overline{\fB}_k^{(\alpha)} t_0^{k-2} 
   \frac{J_\la(X; q, t)}{(1-t)^{|\la|}}   
   + \overline{\fB}_1^{(\alpha)} \frac{1}{t_0} \left (
   \frac{J_\la(X; q, t)}{(1-t)^{|\la|}} - J_\la^{(\alpha)} \right )     \\
&=&\sum_{k \ge 2} c_{\la, k}^{(\alpha)} t_0^{k-2} 
   \cdot \frac{J_\la(X; q, t)}{(1-t)^{|\la|}}
   + c_{\la, 1}^{(\alpha)} \cdot \frac{1}{t_0} \left (
   \frac{J_\la(X; q, t)}{(1-t)^{|\la|}} - J_\la^{(\alpha)} \right ).   
\end{eqnarray*}
Letting $t_0 \to 0$ and using \eqref{TT10210857Z.4}, we conclude that 
$$
  \overline{\fB}_2^{(\alpha)} J_\la^{(\alpha)}
  + \overline{\fB}_1^{(\alpha)} \frac{n(\la') \alpha + n(\la)}{2} J_\la^{(\alpha)}
= c_{\la, 2}^{(\alpha)} J_\la^{(\alpha)}
  + c_{\la, 1}^{(\alpha)} \cdot \frac{n(\la') \alpha + n(\la)}{2} J_\la^{(\alpha)}.
$$
Therefore, $\overline{\fB}_2^{(\alpha)} J_\la^{(\alpha)} 
= c_{\la, 2}^{(\alpha)} \cdot J_\la^{(\alpha)}$,
and \eqref{TT10210857Z.0} holds for $k = 2$ as well.
\end{proof}

We are now ready to prove our first main theorem regarding 
the $k$-th equivariant Chern character operator 
$\fG_k(t_1, t_2) \in \End(\fock').$

\begin{theorem}    \label{TT10210727X}
If $0 \le k \le 2$, then
\begin{eqnarray}    \label{TT10210727X.0}
  \fG_k(t_1, t_2)
= t_2^k t_1^{\delta(\cdot)} 
  \overline{\fB}_k^{(\alpha)}|_{\alpha = -t_1/t_2}
\end{eqnarray}
where $t_1^{\delta(\cdot)} \fa_\la = t_1^{\delta(\la)} \fa_\la 
= t_1^{\ell(\la^-) - \ell(\la^+)} \fa_\la$
for a generalized partition $\la$, and $\la^+$ and $\la^-$ are 
the positive and negative parts of $\la$ respectively.
\end{theorem}
\begin{proof}
Let $c_{\la, k}(t_1, t_2) = (-1)^k/{k!} \cdot \sum_{\square\in D_\lambda} 
(a'(\square)t_1 + \ell'(\square)t_2)^k$. By \eqref{TT09161027Z.2},
\begin{eqnarray}  \label{TT10210727X.1}
  \fG_k(t_1, t_2) {\bf J}^\la(t_1, t_2)
= c_{\la, k}(t_1, t_2) \cdot {\bf J}^\la(t_1, t_2).
\end{eqnarray}

Next, let $\alpha = -t_1/t_2$. Put $\overline{\fB}_k^{(\alpha)} 
= \sum_{|\w \la| = 0} g_{\w \la, k}^{(\alpha)} {\fa_{\w \la}}$.
By \eqref{TT10210857Z.0}, we have
\begin{eqnarray*}   
  t_2^k \sum_{|\w \la| = 0} g_{\w \la, k}^{(\alpha)} {\fa_{\w \la}}
  J_\la^{(\alpha)} 
= c_{\la, k}(t_1, t_2) \cdot J_\la^{(\alpha)}.
\end{eqnarray*}
Put $J_\la^{(\alpha)} = \sum_{|\mu| = |\la|} f_{\la, \mu}^{(\alpha)} 
\fa_{-\mu}\vac$. Then, we obtain
$$
  t_2^k \sum_{|\w \la| = 0} g_{\w \la, k}^{(\alpha)} {\fa_{\w \la}} 
  \sum_{|\w \mu| = |\la|} f_{\la, \w \mu}^{(\alpha)} \fa_{-\w \mu}\vac
= c_{\la, k}(t_1, t_2) \cdot \sum_{|\mu| = |\la|} f_{\la, \mu}^{(\alpha)} 
  \fa_{-\mu}\vac. 
$$
For fixed $\w \la$ with $|\w \la| = 0$ and $\w \mu$ with $|\w \mu| = |\la|$, put
$$
  \fa_{\w \la} \fa_{-\w \mu}\vac
= \sum_{|\mu| = |\la|} h_{\w \la,\w \mu, \mu}^{(\alpha)} \fa_{-\mu}\vac.
$$
Note that $\delta(\w \la) + \ell(\w \mu) = \ell(\mu)$ 
if $h_{\w \la,\w \mu, \mu}^{(\alpha)} \ne 0$. We have
$$
  t_2^k \sum_{|\w \la| = 0} g_{\w \la, k}^{(\alpha)}  
  \sum_{|\w \mu| = |\la|} f_{\la, \w \mu}^{(\alpha)} 
  \sum_{|\mu| = |\la|} h_{\w \la,\w \mu, \mu}^{(\alpha)} \fa_{-\mu}\vac
= c_{\la, k}(t_1, t_2) \cdot \sum_{|\mu| = |\la|} f_{\la, \mu}^{(\alpha)} 
  \fa_{-\mu}\vac. 
$$
So for fixed $\la, k$ and $\mu$ with $|\mu| = |\la|$, we get
\begin{eqnarray}  \label{TT10210727X.2}
  t_2^k \sum_{|\w \la| = 0, |\w \mu| = |\la|} 
  g_{\w \la, k}^{(\alpha)} f_{\la, \w \mu}^{(\alpha)}
  h_{\w \la,\w \mu, \mu}^{(\alpha)}
= c_{\la, k}(t_1, t_2) f_{\la, \mu}^{(\alpha)}. 
\end{eqnarray}
By \eqref{JLaT1T2}, ${\bf J}^\la(t_1, t_2) 
= t_2^{|\la|} t_1^{\ell(\cdot)} J_\la^{(\alpha)}$. Therefore, 
$t_2^k t_1^{\delta(\cdot)} \overline{\fB}_k^{(\alpha)} {\bf J}^\la(t_1, t_2)$
is equal to
\begin{eqnarray*}
& &t_2^k t_1^{\delta(\cdot)} \overline{\fB}_k^{(\alpha)} 
   t_2^{|\la|}t_1^{\ell(\cdot)} J_\la^{(\alpha)}   \\
&=&t_2^{|\la|} t_2^k \sum_{|\w \la| = 0} g_{\w \la, k}^{(\alpha)} 
   t_1^{\delta(\w \la)}{\fa_{\w \la}}
   \sum_{|\w \mu| = |\la|} f_{\la, \w \mu}^{(\alpha)} 
   t_1^{\ell(\w \mu)}\fa_{-\w \mu}\vac   \\
&=&t_2^{|\la|} t_2^k \sum_{|\w \la| = 0} g_{\w \la, k}^{(\alpha)} 
   t_1^{\delta(\w \la)}
   \sum_{|\w \mu| = |\la|} f_{\la, \w \mu}^{(\alpha)} t_1^{\ell(\w \mu)} 
   \sum_{|\mu| = |\la|} h_{\w \la,\w \mu, \mu}^{(\alpha)} \fa_{-\mu}\vac
\end{eqnarray*}
Since $\delta(\w \la) + \ell(\w \mu) = \ell(\mu)$ 
if $h_{\w \la,\w \mu, \mu}^{(\alpha)} \ne 0$, we always have 
$$
  t_1^{\delta(\w \la)} t_1^{\ell(\w \mu)} h_{\w \la,\w \mu, \mu}^{(\alpha)}
= t_1^{\ell(\mu)} h_{\w \la,\w \mu, \mu}^{(\alpha)}.
$$
Hence we conclude that
$t_2^k t_1^{\delta(\cdot)} \overline{\fB}_k^{(\alpha)} {\bf J}^\la(t_1, t_2)$
is equal to
\begin{eqnarray*} 
& &t_2^{|\la|} t_2^k \sum_{|\mu| = |\la|} t_1^{\ell(\mu)}
   \left (\sum_{|\w \la| = 0, |\w \mu| = |\la|} 
   g_{\w \la, k}^{(\alpha)} f_{\la, \w \mu}^{(\alpha)} 
   h_{\w \la,\w \mu, \mu}^{(\alpha)} \right )\fa_{-\mu}\vac  \\
&=&t_2^{|\la|} \sum_{|\mu| = |\la|} t_1^{\ell(\mu)}
   c_{\la, k}(t_1, t_2) f_{\la, \mu}^{(\alpha)} \fa_{-\mu}\vac \\
&=&t_2^{|\la|} t_1^{\ell(\cdot)} c_{\la, k}(t_1, t_2)
   \sum_{|\mu| = |\la|} f_{\la, \mu}^{(\alpha)} \fa_{-\mu}\vac  \\
&=&t_2^{|\la|} t_1^{\ell(\cdot)} c_{\la, k}(t_1, t_2)
   J_\la^{(\alpha)}   \\
&=&c_{\la, k}(t_1, t_2) \cdot {\bf J}^\la(t_1, t_2)
\end{eqnarray*}
where we used \eqref{TT10210727X.2} in the second equality. 
Thus, \eqref{TT10210727X.0} holds.
\end{proof}

Setting $k = 1$ in \eqref{TT10210727X.0} and using \eqref{TT10210857Z.2},
we obtain \eqref{Intro_OP3DefMqt1t2}. Similarly, setting $k = 2$ 
in \eqref{TT10210727X.0} and using \eqref{TT10210857Z.3}, 
we conclude that $\fG_2(t_1, t_2)$ is equal to
$$
\sum_{\ell(\la) = 4, \, |\la| = 0} (-1)^{\ell(\la^+)-1} 
    (t_1t_2)^{\ell(\la^-)-1} \frac{\fa_\la}{\la^!}  
$$
$$
+ \sum_{\ell(\la) = 3, |\la| = 0} \frac{|\la^+|-1}{2} (-1)^{\ell(\la^+)} 
    (t_1t_2)^{\ell(\la^-)-1} (t_1+t_2) \frac{\fa_\la}{\la^!}
$$
\begin{eqnarray}    \label{BB02241101Z}
+ \sum_{i > 0} \left (\frac{2i^2 - 3i +1}{12}(t_1+t_2)^2 
+ \frac{1 - i^2}{12} t_1t_2 \right ) \fa_{-i}\fa_{i}.
\end{eqnarray}
This is the equivariant version of \eqref{Tang5.6} for 
the Chern character operator $\fG_2(\alpha)$.

\section{\bf Okounkov's Conjecture} 
\label{sect_Okounkov}

In this section, we investigate Okounkov's Conjecture 
\cite[Conjecture~2]{Oko} in the equivariant setting of the Hilbert schemes 
$\Xn$ where $X = \C^2$. We will recall the Ext vertex operators of 
Carlsson and Okounkov \cite{CO, Car1}, 
re-interpret Okounkov's reduced series, and prove 
Theorem~\ref{Intro_BB070340X} (= Theorem~\ref{BB070340X}).

Let $\T = (\C^*)^2$ and $L$ be a $\T$-equivariant line bundle over $X = \C^2$. 
Let $\mathbb E_L$ be the virtual vector bundle on $X^{[k]} \times X^{[\ell]}$
whose fiber at $(I, J) \in X^{[k]} \times X^{[\ell]}$ is 
$$
\mathbb E_L|_{(I,J)} = \chi(\mathcal O, L) - \chi(J, I \otimes L).
$$
Let $\fL_m$ be the trivial line bundle on $X$ with a scaling action of $\C^*$ of 
character $z^m$, and let $\Delta_n$ be the diagonal in $\Xn \times \Xn$. Then, 
\begin{eqnarray}   \label{ResOfEToD}
\mathbb E_{\fL_m}|_{\Delta_n} = T_{\Xn, m},
\end{eqnarray}
the tangent bundle $T_\Xn$ with a scaling action of $\C^*$ of character $z^m$. 
By abusing notations, we also use $L$ to denote its first Chern class. Put
\begin{eqnarray}   \label{GammaLzBB0513}
  \Gamma_{\pm}(L, z) 
= \exp \left ( \sum_{n>0} \frac{z^{\mp n}}{n} \fp_{\pm n}(L) \right ).         
\end{eqnarray}

\begin{remark}  \label{SignDiff}
There is a sign difference between the two nondegenerate bilinear forms 
$\langle \cdot, \cdot \rangle$ appearing in the Heisenberg commutation 
relations \eqref{Heisenberg} and \cite[(7)]{Car1}.
So for $n > 0$, our Heisenberg operators $\fp_{-n}(L)$ and $\fp_{n}(-L)$
are equal to the Heisenberg operators $\fa_{-n}(L)$ and $\fa_{n}(L)$ in \cite{Car1}.
Accordingly, our operators $\Gamma_-(L, z)$ and $\Gamma_+(-L, z)$ are equal to
the operators $\Gamma_-(L, z)$ and $\Gamma_+(L, z)$ in \cite{Car1}.
\end{remark}

Let $\Wb(L, z) \in \End(\fock')[[z, z^{-1}]]$ be 
the vertex operator constructed in \cite{CO, Car1} 
where $z$ is a formal variable. 
Following \cite{Car1}, $\Wb(L, z)$ is defined by 
\begin{eqnarray}   \label{def-WLz}
\langle \Wb(L, z) \eta, \xi \rangle 
= z^{\ell - k} \int_{X^{[k]} \times X^{[\ell]}}
   (\eta \otimes \xi) \, c_{k+\ell}(\mathbb E_L)
\end{eqnarray}
for $\eta \in H^*(X^{[k]})'_\T$ and $\xi \in H^*(X^{[\ell]})'_\T$,
where the equivariant integral is defined via localization.
By \cite[Theorem~1]{Car1} (see Remark~\ref{SignDiff}), we have
\begin{eqnarray}   \label{WLz}
\Wb(L, z) = \Gamma_-(L-K_X, z) \, \Gamma_+(-L, z).
\end{eqnarray}
Moreover, noting that our ${\bf J}^\la(t_1, t_2)$ defined by \eqref{JLaT1T2} 
is the same as the $\W J_\la$ defined in \cite[(14)]{Car1}, 
we conclude from \cite[(23)]{Car1} that
\begin{eqnarray}   \label{Car1(23)}
\Wb(\fL_m, z) {\bf J}^\la(t_1, t_2)
= a_{\la, \la} {\bf J}^\la(t_1, t_2) 
  + \sum_{\mu \ne \la} a_{\la, \mu} {\bf J}^\mu(t_1, t_2)
\end{eqnarray}
for some coefficients $a_{\la, \mu}$ with $\mu \ne \la$, 
where $a_{\la, \la}$ is defined by 
\begin{eqnarray}   \label{Car1(23).1}
\frac{\prod_{\square\in D_{\lambda}} 
  \big(m-t_1 a(\square) + t_2 (\ell(\square)+1)\big)
  \big(m+t_1 (a(\square)+1) - t_2 \ell(\square) \big)}
  {\prod_{\square\in D_{\lambda}} 
  \big(-t_1 a(\square) + t_2 (\ell(\square)+1)\big)
  \big(t_1 (a(\square)+1) - t_2 \ell(\square) \big)}.
\end{eqnarray}

In the next lemma, we rewrite $\Wb(\fL_m, z)$ in terms of 
the operators $\Gamma_\pm(z)$.

\begin{lemma}   \label{WLmz}
For the scaling action of $\C^*$ on $\fL_m$, let $H^*_{\C^*}({\rm pt}) = \C[t]$.
Then,
$$
\Wb(\fL_m, z) = \Gamma_-(z)^{mt+t_1+t_2} \, \Gamma_+(z)^{mt/(t_1t_2)}
$$
where the operators $\Gamma_\pm(z)$ are defined in \eqref{Intro_GammaLz}.
\end{lemma}
\begin{proof}
Let $x \in X$ be the origin. 
By the action \eqref{Intro_TT08250207X} and localization, 
$$
K_X = \frac{-(t_1+t_2)}{t_1t_2} [x] = -(t_1+t_2) 1_X.
$$ 
Also, $c_1(\fL_m) = mt 1_X = mt/(t_1t_2) \cdot [x]$. 
By \eqref{WLz} and \eqref{GammaLzBB0513}, we obtain
\begin{eqnarray*}   
& &\Wb(\fL_m, z)   \\
&=&\Gamma_-(\fL_m-K_X, z) \, \Gamma_+(-\fL_m, z)   \\
&=&\exp \left ( \sum_{n>0} \frac{z^{n}}{n} \fp_{-n}((mt+t_1+t_2)1_X) \right )
   \cdot \exp \left ( \sum_{n>0} \frac{z^{-n}}{n} 
   \frac{mt}{t_1t_2} \fp_{n}(-[x]) \right )  \\
&=&\exp \left ( \sum_{n>0} \frac{z^{n}}{n} (mt+t_1+t_2)\fa_{-n} \right )
   \cdot \exp \left ( \sum_{n>0} \frac{z^{-n}}{n} 
   \frac{mt}{t_1t_2}\fa_n \right )    \\
&=&\Gamma_-(z)^{mt+t_1+t_2} \, \Gamma_+(z)^{mt/(t_1t_2)}
\end{eqnarray*}
where we have used \eqref{TT09120954Z} in the third equality.
\end{proof}

The two formal variables $m$ and $t$ can be combined in Lemma~\ref{WLmz}. 
Therefore, in the rest of the paper, we will set $t = 1$ in Lemma~\ref{WLmz}. 
By Lemma~\ref{WLmz}, 
\begin{eqnarray}   \label{WLmzT=1}
\Wb(\fL_m, z) = \Gamma_-(z)^{m+t_1+t_2} \, \Gamma_+(z)^{m/(t_1t_2)}.
\end{eqnarray}
Note that if we set $t_1 = -t_2 = 1$ and $z = 1$ further, 
then we recover the operator
$W(m, 1) = \Gamma_-(1)^m \, \Gamma_+(1)^{-m}$ in \cite[(38)]{Car1}.

Recall the operators 
$\fG_k(t_1, t_2) \in \End(\fock')$ from Definition~\ref{Intro_DefFG}.

\begin{definition}   \label{trace}
Let $\mathfrak n$ be the number-of-points operator, i.e., 
$\mathfrak n|_{H^*(X^{[n]})'_\T} =  n \, {\rm Id}$.
Let $\alpha = -t_1/t_2$.
For $N, k_1, \ldots, k_N \ge 0$, define the traces
\begin{eqnarray}
   \langle \ch_{k_1} \cdots \ch_{k_N} \rangle_{m, t_1, t_2}
&=&\Tr_{\fock'} q^{\mathfrak n} \, \Wb(\fL_m, z) \prod_{j=1}^N 
       \fG_{k_j}(t_1, t_2),     \label{trace.01}   \\
   {[\ch_{k_1} \cdots \ch_{k_N}]_{m, t_1, t_2}}
&=&\Tr_{\fock'} q^{\mathfrak n} \, \Wb(\fL_m, z) \prod_{j=1}^N
       t_2^{k_j} t_1^{\delta(\cdot)} \overline{\fB}_{k_j}^{(\alpha)}
       \label{trace.02}
\end{eqnarray}
where $\overline{\fB}_k^{(\alpha)} 
= \Coe_{w^k} \overline{\fB}(\w q, \w t^{-1})$ is from 
\eqref{Intro_OverlinefBkAlpha} with $\w q = e^{\alpha w}$ and $\w t = e^{w}$.
\end{definition} 

\begin{remark}   \label{rmk_trace}
\begin{enumerate}
\item[{\rm (i)}] By considering the conformal weights of 
the Heisenberg operators in $\Wb(\fL_m, z)$, we see that 
$\langle \ch_{k_1} \cdots \ch_{k_N} \rangle_{m, t_1, t_2}$ and 
$[\ch_{k_1} \cdots \ch_{k_N}]_{m, t_1, t_2}$ are independent of 
the formal variable $z \ne 0$. In particular, setting $z=1$ in 
\eqref{trace.01} and \eqref{trace.02}, we obtain from \eqref{WLmzT=1} that 
\begin{eqnarray}   
   {\langle \ch_{k_1} \cdots \ch_{k_N} \rangle_{m, t_1, t_2}}
&=&\Tr q^{\mathfrak n} \, 
       \Gamma_-(1)^{m+t_1+t_2} \, \Gamma_+(1)^{m/(t_1t_2)} \prod_{j=1}^N 
       \fG_{k_j}(t_1, t_2),     \label{trace.0100}   \\
   {[\ch_{k_1} \cdots \ch_{k_N}]_{m, t_1, t_2}}
&=&\Tr q^{\mathfrak n} \, 
       \Gamma_-(1)^{m+t_1+t_2} \, \Gamma_+(1)^{m/(t_1t_2)} \prod_{j=1}^N 
       t_2^{k_j} t_1^{\delta(\cdot)} \overline{\fB}_{k_j}^{(\alpha)}.
       \label{trace.0200}
\end{eqnarray}

\item[{\rm (ii)}] 
Using \eqref{def-WLz}, \eqref{ResOfEToD} and \eqref{Intro_DefFG.2},
we conclude as in \cite{Car1, Car2} and \cite[Lemma~5.2]{Qin1} that 
$\langle \ch_{k_1} \cdots \ch_{k_N} \rangle_{m, t_1, t_2}$ is equal to 
the generating series
$$
\sum_{n \ge 0} q^n \int_{X^{[n]}} \left (\prod_{j=1}^N \ch_{k_j}^\T
  \big ( \On \big ) \right ) e\big (T_{\Xn, m} \big )
$$
of equivariant integrals over the Hilbert schemes $X^{[n]}$.

\item[{\rm (iii)}]
By Theorem~\ref{TT10210727X}, if $k_1, \ldots, k_N \in 
\{0, 1, 2\}$, then 
\begin{eqnarray}   \label{rmk_trace.0}
  \langle \ch_{k_1} \cdots \ch_{k_N} \rangle_{m, t_1, t_2}
= [\ch_{k_1} \cdots \ch_{k_N}]_{m, t_1, t_2}.
\end{eqnarray}
Also, by definition, \eqref{rmk_trace.0} is trivially true when $N = 0$.
\end{enumerate}
\end{remark}  

We may express $\langle \ch_{k_1} \cdots \ch_{k_N} \rangle_{m, t_1, t_2}$
in terms of (usual) partitions.

\begin{lemma}   \label{BB11241134Z}
The series
$\langle \ch_{k_1} \cdots \ch_{k_N} \rangle_{m, t_1, t_2}$ is equal to
\begin{eqnarray}   \label{BB11241134Z.01}
\sum_\la q^{|\la|} a_{\la, \la} \prod_{j=1}^N 
\sum_{\square\in D_\lambda} \frac{(-1)^{k_j}}{k_j!} 
     (a'(\square)t_1 + \ell'(\square)t_2)^{k_j}
\end{eqnarray}
where $a_{\la, \la} \in \Q(t_1, t_2)[m]$ is from \eqref{Car1(23).1}.
In particular,
\begin{eqnarray}   \label{BB11241134Z.02}
\langle \ch_{k_1} \cdots \ch_{k_N} \rangle_{m, t_1, t_2}
\in \Q(t_1, t_2)[[m, q]].
\end{eqnarray}
\end{lemma}
\begin{proof}
Follows from \eqref{TT09161027Z.2}, \eqref{Car1(23)} and the definition of $\langle \ch_{k_1} \cdots \ch_{k_N} \rangle_{m, t_1, t_2}$.
\end{proof}

Unfortunately, it is almost impossible to study 
$\langle \ch_{k_1} \cdots \ch_{k_N} \rangle_{m, t_1, t_2}$ 
from \eqref{BB11241134Z.01}. Next, we introduce the reduced series. 
Whenever necessary, we will pretend that the expressions 
$m+t_1+t_2$ and $m/(t_1t_2)$ are integers. 
By \cite[Lemma~3]{Car2}, 
\begin{eqnarray}   \label{Lemma3InCar2}
\Tr_{\fock'} q^{\mathfrak n} \, \prod_{j=1}^\ell \Gamma_-(y_j)^{b_j}
\prod_{i=1}^k \Gamma_+(x_i)^{a_i} 
= (q;q)_\infty^{-1} \cdot \prod_{i, j} (y_jx_i^{-1}q;q)_\infty^{-a_ib_j}.
\end{eqnarray}
Applying this to \eqref{trace.0100} and \eqref{trace.0200}, we see that 
\begin{eqnarray}   \label{ExampleF01.1}
  \langle \rangle_{m, t_1, t_2} 
= [ \, ]_{m, t_1, t_2}
= (q;q)_\infty^{-1-(m+t_1+t_2)m/(t_1t_2)}.
\end{eqnarray}
Following \cite{Oko}, we define the {\it reduced series} to be
\begin{eqnarray}   
\langle \ch_{k_1} \cdots \ch_{k_N} \rangle_{m, t_1, t_2}'
&=&\frac{\langle \ch_{k_1} \cdots \ch_{k_N} 
   \rangle_{m, t_1, t_2}}{\langle \rangle_{m, t_1, t_2}}    \nonumber   \\
&=&(q;q)_\infty^{1+(m+t_1+t_2)m/(t_1t_2)} 
   \langle \ch_{k_1} \cdots \ch_{k_N} \rangle_{m, t_1, t_2},    
   \label{Reduced.1}   \\
{[\ch_{k_1} \cdots \ch_{k_N}]_{m, t_1, t_2}'}
&=&\frac{[\ch_{k_1} \cdots \ch_{k_N}]_{m, t_1, t_2}}
   {[ \, ]_{m, t_1, t_2}}    \nonumber   \\
&=&(q;q)_\infty^{1+(m+t_1+t_2)m/(t_1t_2)} 
   [\ch_{k_1} \cdots \ch_{k_N}]_{m, t_1, t_2}.    \label{Reduced.2}   
\end{eqnarray}
In view of \eqref{BB11241134Z.02}, we conclude that 
\begin{eqnarray}   \label{BB11241134Z.03}
\langle \ch_{k_1} \cdots \ch_{k_N} \rangle_{m, t_1, t_2}'
\in \Q(t_1, t_2)[[q, m]].
\end{eqnarray}

Using \eqref{Intro_G0t1t2}, \eqref{Intro_OP3DefMqt1t2} and 
\eqref{trace.0100}, we are able to compute the reduced series 
$\langle \ch_0 \rangle_{m, t_1, t_2}'$ and 
$\langle \ch_1 \rangle_{m, t_1, t_2}'$ directly. We have
\begin{eqnarray}  
   \langle \ch_0 \rangle_{m, t_1, t_2}'
&=&\left (1 + \frac{(m+t_1+t_2)m}{t_1t_2} \right ) [2], \label{BB11240203X} \\
   \langle \ch_1 \rangle_{m, t_1, t_2}'
&=&\left (t_1+t_2 + \frac{(m+t_1+t_2)(t_1+t_2)m}{t_1t_2} \right )
   \frac{[2] - 2[3]}{2}    \label{BB11240224X}
\end{eqnarray}
(see also \cite[Subsection~2.6]{Oko} for \eqref{BB11240224X}).
When $m=1$, these are the equivariant versions of 
the formulas (6.2) and (6.4) in \cite{AQ} respectively. 

\begin{definition}   \label{AA05210933Z}
For operators $\mathfrak f_1, \ldots, \mathfrak f_N$ and a subset $S$ of 
${\underline N} = \{1, \ldots, N\}$, 
if $S = \{s_1, \ldots, s_\ell \}$ with $s_1 < \ldots < s_\ell$,
then the product $\prod_{j \in S} \mathfrak f_j$ is defined to be
$$
\prod_{j=1}^{\ell} 
\mathfrak f_{s_j} = \mathfrak f_{s_1} \cdots \mathfrak f_{s_\ell}.
$$
\end{definition} 

Our goals are to study $[\ch_{k_1} \cdots \ch_{k_N}]_{m, t_1, t_2}$ 
and then to draw consequences about the reduced series
$\langle \ch_{k_1} \cdots \ch_{k_N} \rangle_{m, t_1, t_2}'$ 
by using Remark~\ref{rmk_trace}~(iii). 

\begin{lemma}   \label{BB05270929Z}
Let $\alpha = -t_1/t_2$. 
Then, $[\ch_{k_1} \cdots \ch_{k_N}]_{m, t_1, t_2}$ is equal to
$$
t_2^{\sum_{j=1}^N k_j} (q;q)_\infty^{-1-\frac{(m+t_1+t_2)m}{t_1t_2}} \cdot 
\Coe_{z_1^{k_1} y_1^0 \cdots z_N^{k_N} y_N^0}
\prod_{j=1}^N \frac{1}{(1-\w q_j)(1-\w t_j^{-1})}
$$
$$ 
\sum_{S \subset {\underline N}} (-1)^{|S|} \prod_{j \in S} 
\frac{(q)_\infty (\w t_j^{-1} \w q_j q)_\infty}
     {(\w q_j q)_\infty (\w t_j^{-1}q)_\infty}
\cdot \prod_{\ell, j \in S, \, \ell < j} 
\frac{(\w t_j y_j(\w t_\ell y_\ell)^{-1})_\infty
      (y_j(\w t_\ell \w q_\ell^{-1} y_\ell)^{-1})_\infty}
     {(\w t_j y_j(\w t_\ell \w q_\ell^{-1} y_\ell)^{-1})_\infty
      (y_j(\w t_\ell y_\ell)^{-1})_\infty}
$$
$$
\cdot \prod_{\ell, j \in S, \, \ell > j} 
\frac{(\w t_j y_j(\w t_\ell y_\ell)^{-1}q)_\infty
      (y_j(\w t_\ell \w q_\ell^{-1} y_\ell)^{-1}q)_\infty}
     {(\w t_j y_j(\w t_\ell \w q_\ell^{-1} y_\ell)^{-1}q)_\infty
      (y_j(\w t_\ell y_\ell)^{-1}q)_\infty} 
$$
$$
\cdot 
\prod_{j \in S} ((\w t_j \w q_j^{-1} y_j)^{-1}q)_\infty^{-(m+t_1+t_2)/t_1}
((\w t_j y_j)^{-1}q)_\infty^{(m+t_1+t_2)/t_1}
(\w t_j y_j)_\infty^{-m/t_2} (y_j)_\infty^{m/t_2}
$$
where $\w q_j = e^{\alpha z_j}$ and $\w t_j = e^{z_j}$ 
for every $1 \le j \le N$.
\end{lemma}
\begin{proof}
Recall from \eqref{Intro_OverlinefBkAlpha} that $\overline{\fB}_k^{(\alpha)} 
= \Coe_{z^k} \overline{\fB}(\w q, \w t^{-1})$ 
where $\w q = e^{\alpha z}$ and $\w t = e^{z}$. 
By \eqref{TT10180917Z.0} and \eqref{CW1}, 
$t_1^{\delta(\cdot)} \overline{\fB}_k^{(\alpha)}$ is equal to
\begin{eqnarray*}    
& &t_1^{\delta(\cdot)} \Coe_{z^k} \, \frac{1}{(1-\w q)(1-\w t^{-1})} 
  \big (1 - V_0(y; \w t, \w t^{-1}, 1, \w q \w t^{-1})\big )  \\
&=&t_1^{\delta(\cdot)} \Coe_{z^ky^0} \, \frac{1}{(1-\w q)(1-\w t^{-1})} 
  \big (1 - V(y; \w t, \w t^{-1}, 1, \w q \w t^{-1})\big )  \\
&=&t_1^{\delta(\cdot)} \Coe_{z^ky^0} \, \frac{1}{(1-\w q)(1-\w t^{-1})} 
   \big (1 - \Gamma_-(\w t y) \Gamma_-(y)^{-1} 
   \Gamma_+(\w t \w q^{-1}y) \Gamma_+(\w t y)^{-1}\big )   \\
&=&\Coe_{z^k y^0} \frac{1}{(1-\w q)(1-\w t^{-1})} 
   \big (1 - \Gamma_-(\w t y)^{t_1} \Gamma_-(y)^{-t_1} 
   \Gamma_+(\w t \w q^{-1}y)^{t_1^{-1}} \Gamma_+(\w t y)^{-t_1^{-1}}\big ).
\end{eqnarray*}
Combining with \eqref{trace.0200}, we see that
$[\ch_{k_1} \cdots \ch_{k_N}]_{m, t_1, t_2}$ is equal to
$$
t_2^{\sum_{j=1}^N k_j} \cdot 
\Coe_{z_1^{k_1} y_1^0 \cdots z_N^{k_N} y_N^0}
\prod_{j=1}^N \frac{1}{(1-\w q_j)(1-\w t_j^{-1})} 
\Tr \, q^{\mathfrak n} \, \Gamma_-(1)^{m+t_1+t_2} 
\Gamma_+(1)^{\frac{m}{t_1t_2}}
$$
\begin{eqnarray}   \label{BB07200802X}
\prod_{j=1}^N \big (1 - \Gamma_-(\w t_j y_j)^{t_1} \Gamma_-(y_j)^{-t_1} 
  \Gamma_+(\w t_j \w q_j^{-1} y_j)^{t_1^{-1}} 
  \Gamma_+(\w t_j y_j)^{-t_1^{-1}}\big ).
\end{eqnarray}
In view of the notation $\prod_{j \in S}$ in Definition~\ref{AA05210933Z}, 
$[\ch_{k_1} \cdots \ch_{k_N}]_{m, t_1, t_2}$ is equal to
\begin{eqnarray}   \label{BB05270929Z.2}
t_2^{\sum_{j=1}^N k_j} \cdot 
\Coe_{z_1^{k_1} y_1^0 \cdots z_N^{k_N} y_N^0}
\prod_{j=1}^N \frac{1}{(1-\w q_j)(1-\w t_j^{-1})} 
\cdot \sum_{S \subset {\underline N}} (-1)^{|S|}
\end{eqnarray}
$$
\Tr \, q^{\mathfrak n} \, \Gamma_-(1)^{m+t_1+t_2} 
\Gamma_+(1)^{\frac{m}{t_1t_2}}
\prod_{j \in S} \Gamma_-(\w t_j y_j)^{t_1} \Gamma_-(y_j)^{-t_1} 
  \Gamma_+(\w t_j \w q_j^{-1} y_j)^{\frac1{t_1}} 
  \Gamma_+(\w t_j y_j)^{-\frac1{t_1}}.
$$

By Lemma~\ref{GammaPMComm}~(ii), 
the trace below line \eqref{BB05270929Z.2} is equal to
$$
\prod_{j \in S} (1-\w t_j y_j)^{-m/t_2} (1-y_j)^{m/t_2} 
\cdot \prod_{\ell, j \in S, \, \ell < j} 
\frac{(1-\w t_j y_j(\w t_\ell y_\ell)^{-1})
       (1-y_j(\w t_\ell \w q_\ell^{-1} y_\ell)^{-1})}
      {(1-\w t_j y_j(\w t_\ell \w q_\ell^{-1} y_\ell)^{-1})
       (1-y_j(\w t_\ell y_\ell)^{-1})}
$$
$$
\Tr q^{\mathfrak n} \Gamma_-(1)^{m+t_1+t_2} 
\prod_{j \in S} \Gamma_-(\w t_j y_j)^{t_1} \Gamma_-(y_j)^{-t_1} 
\cdot \Gamma_+(1)^{\frac{m}{t_1t_2}} 
\prod_{j \in S} \Gamma_+(\w t_j \w q_j^{-1} y_j)^{\frac1{t_1}} 
\Gamma_+(\w t_j y_j)^{-\frac1{t_1}}.
$$
Therefore, using \eqref{Lemma3InCar2}, we see that 
the trace below line \eqref{BB05270929Z.2} is equal to
$$
(q;q)_\infty^{-1-(m+t_1+t_2)m/(t_1t_2)}
\cdot \prod_{\ell, j \in S, \, \ell < j} 
\frac{(\w t_j y_j(\w t_\ell y_\ell)^{-1})_\infty
      (y_j(\w t_\ell \w q_\ell^{-1} y_\ell)^{-1})_\infty}
     {(\w t_j y_j(\w t_\ell \w q_\ell^{-1} y_\ell)^{-1})_\infty
      (y_j(\w t_\ell y_\ell)^{-1})_\infty}
$$
$$
\cdot \prod_{\ell, j \in S, \, \ell > j} 
\frac{(\w t_j y_j(\w t_\ell y_\ell)^{-1}q)_\infty
      (y_j(\w t_\ell \w q_\ell^{-1} y_\ell)^{-1}q)_\infty}
     {(\w t_j y_j(\w t_\ell \w q_\ell^{-1} y_\ell)^{-1}q)_\infty
      (y_j(\w t_\ell y_\ell)^{-1}q)_\infty} 
\cdot \prod_{j \in S} 
\frac{(q)_\infty (\w t_j^{-1} \w q_j q)_\infty}
     {(\w q_j q)_\infty (\w t_j^{-1}q)_\infty}
$$
$$
\cdot 
\prod_{j \in S} ((\w t_j \w q_j^{-1} y_j)^{-1}q)_\infty^{-(m+t_1+t_2)/t_1}
((\w t_j y_j)^{-1}q)_\infty^{(m+t_1+t_2)/t_1}
(\w t_j y_j)_\infty^{-m/t_2} (y_j)_\infty^{m/t_2}.
$$
Combining with line \eqref{BB05270929Z.2}, 
we complete the proof of the lemma.
\end{proof}

\begin{proposition}   \label{BB06111137Z}
Let $\alpha = -t_1/t_2$. Then, 
$[\ch_{k_1} \cdots \ch_{k_N}]_{m, t_1, t_2}'$ is equal to
$$
t_2^{\sum_{j=1}^N k_j} \cdot \sum_{S \subset {\underline N}} (-1)^{|S|} 
\Coe_{z_1^{k_1} y_1^0 \cdots z_N^{k_N} y_N^0} 
\prod_{j=1}^N \frac{1}{(1-\w q_j)(1-\w t_j^{-1})} \cdot
\prod_{j \in S} 
\frac{(q)_\infty (\w q_j \w t_j^{-1}q)_\infty}
{(\w q_j q)_\infty
(\w t_j^{-1}q)_\infty}
$$
$$
\cdot \prod_{\substack{\ell, j \in S\\\ell < j}} 
\frac{(\w t_j y_j(\w t_\ell y_\ell)^{-1})_\infty
(y_j(\w t_\ell \w q_\ell^{-1} y_\ell)^{-1})_\infty}
{(\w t_j y_j(\w t_\ell \w q_\ell^{-1} y_\ell)^{-1})_\infty
(y_j(\w t_\ell y_\ell)^{-1})_\infty}
\frac{(\w t_\ell y_\ell(\w t_j y_j)^{-1}q)_\infty
(y_\ell(\w t_j \w q_j^{-1} y_j)^{-1}q)_\infty}
{(\w t_\ell y_\ell(\w t_j \w q_j^{-1} y_j)^{-1}q)_\infty
(y_\ell(\w t_j y_j)^{-1}q)_\infty}
$$
$$
\cdot \prod_{j \in S} 
\left ( \frac{(\w t_j y_j)_\infty}{(y_j)_\infty} \right )^{-m/t_2}
\left ( \frac{((\w t_j \w q_j^{-1} y_j)^{-1}q)_\infty}
  {((\w t_j y_j)^{-1}q)_\infty} \right )^{-(m+t_1+t_2)/t_1}.
$$
where $\w q_j = e^{\alpha z_j}$ and $\w t_j = e^{z_j}$ 
for every $1 \le j \le N$.
\end{proposition}
\begin{proof}
Follows immediately from Lemma~\ref{BB05270929Z} and \eqref{ExampleF01.1}.
\end{proof}

\begin{theorem} \label{BB11250927Z}
Let $k_1, \ldots, k_N$ be nonnegative integers. Then, 
$$
[\ch_{k_1} \cdots \ch_{k_N}]_{m, t_1, t_2}' \in \BD(t_1, t_2)[m]
$$
with both degree (in $m$) and weight at most $\sum_{i=1}^N (k_i + 2)$.
\end{theorem}
\begin{proof}
Let $\alpha = -t_1/t_2$. For $1 \le j \le N$, 
let $\w q_j = e^{\alpha z_j}$ and $\w t_j = e^{z_j}$. 
Replacing every $\w t_j y_j$ by $y_j$, we see from 
Proposition~\ref{BB06111137Z} that 
$[\ch_{k_1} \cdots \ch_{k_N}]_{m, t_1, t_2}'$ is equal to
$$
t_2^{\sum_{j=1}^N k_j} \cdot \sum_{S \subset {\underline N}} (-1)^{|S|} 
\Coe_{z_1^{k_1} \cdots z_N^{k_N}} 
\prod_{j=1}^N \frac{1}{(1-\w q_j)(1-\w t_j^{-1})} \cdot
\prod_{j \in S} \frac{(q)_\infty (\w q_j \w t_j^{-1}q)_\infty}
{(\w q_j q)_\infty (\w t_j^{-1}q)_\infty}
$$
$$
\cdot \Coe_{\prod_{j \in S} y_j^0}
\prod_{\substack{\ell, j \in S\\\ell < j}} 
\frac{(\w q_\ell \w t_j^{-1} y_j y_\ell^{-1})_\infty 
(y_jy_\ell^{-1})_\infty}{(\w q_\ell y_j y_\ell^{-1})_\infty
(\w t_j^{-1} y_j y_\ell^{-1})_\infty}
\frac{(\w q_j \w t_\ell^{-1} y_\ell y_j^{-1}q)_\infty 
(y_\ell y_j^{-1}q)_\infty}
{(\w q_j y_\ell y_j^{-1}q)_\infty (\w t_\ell^{-1} y_\ell y_j^{-1}q)_\infty}
$$
\begin{eqnarray}   \label{BB11250927Z.1}
\cdot \prod_{j \in S} 
\left ( \frac{(\w t_j^{-1} y_j)_\infty}{(y_j)_\infty} \right )^{m/t_2}
\left ( \frac{(\w q_j y_j^{-1}q)_\infty}
  {(y_j^{-1}q)_\infty} \right )^{-(m+t_1+t_2)/t_1}.
\end{eqnarray}   
Since $\w q_j = e^{\alpha z_j}$ and $\w t_j = e^{z_j}$ for 
every $1 \le j \le N$, we have
\begin{eqnarray}   \label{BB11250927Z.2}
    \prod_{j=1}^N \frac{1}{(1-\w q_j)(1-\w t_j^{-1})}
\in \prod_{j=1}^N \frac{1}{\alpha z_j^2} \cdot 
    \Q[[z_j, \alpha z_j|1 \le j \le N]].
\end{eqnarray}
By \cite[Theorem~1.1 and Theorem~1.3 (i)]{Qin3}, 
$
[\ch_{k_1} \cdots \ch_{k_N}]_{m, t_1, t_2}' \in \BD(t_1, t_2)[m]
$
with both degree (in $m$) and weight at most $\sum_{i=1}^N (k_i + 2)$.
\end{proof}

Next, we study the special cases 
$[\ch_{k_1} \cdots \ch_{k_N}]_{0, t_1, t_2}'$
and $\langle \ch_{k_1} \cdots \ch_{k_N} \rangle_{0, t_1, t_2}'$.
Even though we do not know how to compute 
$\langle \ch_{k_1} \cdots \ch_{k_N} \rangle_{m, t_1, t_2}'$ 
for a general $m$, the special case 
$\langle \ch_{k_1} \cdots \ch_{k_N} \rangle_{0, t_1, t_2}'$ 
can be calculated rather easily.

\begin{lemma}   \label{CaseM=0}
Let $\alpha = -t_1/t_2$. Let $\w q_j = e^{\alpha z_j}$ and 
$\w t_j = e^{z_j}$ for every $1 \le j \le N$. 
Fix integers $k_1, \ldots, k_N \ge 0$.
Then the following three expressions are equal:
\begin{enumerate}
\item[{\rm (i)}]
$\langle \ch_{k_1} \cdots \ch_{k_N} \rangle_{0, t_1, t_2}'$;

\item[{\rm (ii)}]
$[\ch_{k_1} \cdots \ch_{k_N}]_{0, t_1, t_2}'$;

\item[{\rm (iii)}]
$\displaystyle{
t_2^{\sum_{j=1}^N k_j} \cdot \sum_{S \subset {\underline N}} (-1)^{|S|} 
\Coe_{z_1^{k_1} \cdots z_N^{k_N}} 
\prod_{j=1}^N \frac{1}{(1-\w q_j)(1-\w t_j^{-1})} \cdot
\prod_{j \in S} \frac{(q)_\infty (\w q_j \w t_j^{-1}q)_\infty}
{(\w q_j q)_\infty (\w t_j^{-1}q)_\infty}
}$
\begin{eqnarray}    \label{CaseM=0.0}
\cdot \Coe_{\prod_{j \in S} y_j^0}
\prod_{\substack{\ell, j \in S\\\ell < j}} 
\frac{(\w q_\ell \w t_j^{-1} y_j y_\ell^{-1})_\infty 
(y_jy_\ell^{-1})_\infty}{(\w q_\ell y_j y_\ell^{-1})_\infty
(\w t_j^{-1} y_j y_\ell^{-1})_\infty}
\frac{(\w q_j \w t_\ell^{-1} y_\ell y_j^{-1}q)_\infty 
(y_\ell y_j^{-1}q)_\infty}{(\w q_j y_\ell y_j^{-1}q)_\infty 
(\w t_\ell^{-1} y_\ell y_j^{-1}q)_\infty}.
\end{eqnarray}   
\end{enumerate}
\end{lemma}
\begin{proof}
The equality (ii) = (iii) follows from \eqref{BB11250927Z.1} by setting $m=0$ 
and by noticing that the only contribution of the factor 
$$
\prod_{j \in S} 
\left ( \frac{(\w q_j y_j^{-1}q)_\infty}
  {(y_j^{-1}q)_\infty} \right )^{-(t_1+t_2)/t_1}
$$
to the operation $\Coe_{\prod_{j \in S} y_j^0}$ in  
$[\ch_{k_1} \cdots \ch_{k_N}]_{0, t_1, t_2}'$ is the constant term $1$.

Next, we prove (i) = (ii). 
Since the conformal weight of $\fG_{k_j}(t_1, t_2)$ is $0$, 
\begin{eqnarray}   \label{BB09031203X}
  \langle \ch_{k_1} \cdots \ch_{k_N} \rangle_{0, t_1, t_2}
&=&\Tr_{\fock'} q^{\mathfrak n} \, 
       \Gamma_-(1)^{t_1+t_2} \prod_{j=1}^N \fG_{k_j}(t_1, t_2)
       \nonumber  \\
&=&\Tr_{\fock'} q^{\mathfrak n} \prod_{j=1}^N \fG_{k_j}(t_1, t_2)
\end{eqnarray}
by \eqref{trace.0100}.
Put $\w q_j = e^{-t_1w_j}$ and $\w t_j = e^{t_2w_j}$. 
By \eqref{Intro_DefFG.2} and Lemma~\ref{TT10180917Z}, 
\begin{eqnarray}   \label{BB07200820X}
   \langle \ch_{k_1} \cdots \ch_{k_N} \rangle_{0, t_1, t_2}
&=&\Coe_{w_1^{k_1} \cdots w_N^{k_N}} 
     \sum_{n \ge 0} q^n \sum_{\la \vdash n} \prod_{j=1}^N 
     \sum_{\square\in D_\lambda} 
     \w q_j^{a'(\square)} (\w t_j^{-1})^{\ell'(\square)}  \qquad \\
&=&\Coe_{w_1^{k_1} \cdots w_N^{k_N}} \Tr \, q^{\mathfrak n} 
      \prod_{j=1}^N \overline{\fB}(\w q_j, \w t_j^{-1})   \nonumber  
\end{eqnarray}
where the last trace 
is taken over $\Lambda_\C \otimes_\C \C(t_1, t_2) \cong \fock'$.
By \eqref{TT10180917Z.0} and \eqref{CW1}, 
$\langle \ch_{k_1} \cdots \ch_{k_N} \rangle_{0, t_1, t_2}$ is equal to
\begin{eqnarray*}  
& &\Coe_{w_1^{k_1}y_1^0 \cdots w_N^{k_N}y_N^0} 
      \prod_{j=1}^N \frac{1}{(1-\w q_j)(1-\w t_j^{-1})} \Tr \, q^{\mathfrak n} 
      \prod_{j=1}^N 
      \big (1 - V(y_j; \w t_j, \w t_j^{-1}, 1, \w q_j \w t_j^{-1})\big )  \\
&=&\Coe_{w_1^{k_1}y_1^0 \cdots w_N^{k_N}y_N^0} 
      \prod_{j=1}^N \frac{1}{(1-\w q_j)(1-\w t_j^{-1})} \\
& &\Tr \, q^{\mathfrak n} 
      \prod_{j=1}^N \big (1 - \Gamma_-(\w t_j y_j) \Gamma_-(y_j)^{-1} 
      \Gamma_+(\w t_j \w q_j^{-1} y_j) 
      \Gamma_+(\w t_j y_j)^{-1}\big ).
\end{eqnarray*}
By Lemma~\ref{GammaPMComm}~(ii) and \eqref{Lemma3InCar2}, 
$\langle \ch_{k_1} \cdots \ch_{k_N} \rangle_{0, t_1, t_2}$ is equal to
$$
\Coe_{w_1^{k_1}y_1^0 \cdots w_N^{k_N}y_N^0} 
   \prod_{j=1}^N \frac{1}{(1-\w q_j)(1-\w t_j^{-1})} 
$$
$$
\Tr \, q^{\mathfrak n} 
   \prod_{j=1}^N \big (1 - \Gamma_-(\w t_j y_j)^{t_1} \Gamma_-(y_j)^{-t_1} 
   \Gamma_+(\w t_j \w q_j^{-1} y_j)^{t_1^{-1}} 
   \Gamma_+(\w t_j y_j)^{-t_1^{-1}}\big ).    
$$
Put $w_j = z_j/t_2$ so that $\w q_j = e^{\alpha z_j}$ and 
$\w t_j = e^{z_j}$. Then, 
\begin{eqnarray}    \label{BB07200817X}
& &\langle \ch_{k_1} \cdots \ch_{k_N} \rangle_{0, t_1, t_2} \nonumber  \\
&=&t_2^{\sum_{j=1}^N k_j} \cdot 
   \Coe_{z_1^{k_1} y_1^0 \cdots z_N^{k_N} y_N^0}
   \prod_{j=1}^N \frac{1}{(1-\w q_j)(1-\w t_j^{-1})}   \nonumber    \\
& &\Tr \, q^{\mathfrak n} 
   \prod_{j=1}^N \big (1 - \Gamma_-(\w t_j y_j)^{t_1} \Gamma_-(y_j)^{-t_1} 
   \Gamma_+(\w t_j \w q_j^{-1} y_j)^{t_1^{-1}} 
   \Gamma_+(\w t_j y_j)^{-t_1^{-1}}\big ).
\end{eqnarray}
Since the conformal weight of $\overline{\fB}_k^{(\alpha)}$ is $0$,
we see from \eqref{BB07200802X} that 
$[\ch_{k_1} \cdots \ch_{k_N}]_{0, t_1, t_2}$ is also given by 
\eqref{BB07200817X}. 
So $\langle \ch_{k_1} \cdots \ch_{k_N} \rangle_{0, t_1, t_2} 
= [\ch_{k_1} \cdots \ch_{k_N}]_{0, t_1, t_2}$. 
By \eqref{ExampleF01.1}, we have $\langle \rangle_{0, t_1, t_2} 
= [ \, ]_{0, t_1, t_2} = (q;q)_\infty^{-1}.$ Hence (i) = (ii).
\end{proof}

The following (= Theorem~\ref{Intro_BB070340X}) is the main result 
in this section. 

\begin{theorem}     \label{BB070340X}
\begin{enumerate}
\item[{\rm (i)}]
Let $k_1, \ldots, k_N$ be nonnegative integers. Then, 
$$
\langle \ch_{k_1} \cdots \ch_{k_N} \rangle_{0, t_1, t_2}' 
\in \qBD[t_1, t_2]
$$
is a degree-$\sum_{i=1}^N k_i$ symmetric homogeneous polynomial in 
$t_1$ and $t_2$ whose coefficients have weights at most 
$\sum_{i=1}^N (k_i+2)$.

\item[{\rm (ii)}]
Let $k \ge 0$. Then, Okounkov's Conjecture \cite[Conjecture~2]{Oko} 
holds for $\langle \ch_{k} \rangle_{0, t_1, t_2}'$.
More precisely, $\langle \ch_{k} \rangle_{0, t_1, t_2}' 
\in \qMZV[t_1, t_2]$ is a degree-$k$ symmetric homogeneous polynomial in 
$t_1$ and $t_2$ whose coefficients have weights at most 
$(k+2)$, and the coefficient of $t_1^k$ (and $t_2^k$) in 
$\langle \ch_{k} \rangle_{0, t_1, t_2}'$ is equal to 
$$
(-1)^k \cdot [k+2] + \big (\text{terms with weights $< (k+2)$} \big ).
$$

\item[{\rm (iii)}] 
If $k_1, \ldots, k_N \in \{0, 1, 2\}$, then 
$
\langle \ch_{k_1} \cdots \ch_{k_N} \rangle_{m, t_1, t_2}'
\in \BD(t_1, t_2)[m]
$
with both degree (in $m$) and weight at most $\sum_{i=1}^N (k_i + 2)$.
\end{enumerate}
\end{theorem}
\begin{proof}
(i) Let $\alpha = -t_1/t_2$. Let $\w q_j = e^{\alpha z_j}$ and 
$\w t_j = e^{z_j}$ for every $1 \le j \le N$. By Lemma~\ref{CaseM=0}, 
$\langle \ch_{k_1} \cdots \ch_{k_N} \rangle_{0, t_1, t_2}'
= [\ch_{k_1} \cdots \ch_{k_N}]_{0, t_1, t_2}$ is equal to
$$
t_2^{\sum_{j=1}^N k_j} \cdot \sum_{S \subset {\underline N}} (-1)^{|S|} 
\Coe_{z_1^{k_1} \cdots z_N^{k_N}} 
\prod_{j=1}^N \frac{1}{(1-\w q_j)(1-\w t_j^{-1})} \cdot
\prod_{j \in S} \frac{(q)_\infty (\w q_j \w t_j^{-1}q)_\infty}
{(\w q_j q)_\infty (\w t_j^{-1}q)_\infty}
$$
\begin{eqnarray}    \label{BB070340X.1}
\cdot \Coe_{\prod_{j \in S} y_j^0}
\prod_{\substack{\ell, j \in S\\\ell < j}} 
\frac{(\w q_\ell \w t_j^{-1} y_j y_\ell^{-1})_\infty 
(y_jy_\ell^{-1})_\infty}{(\w q_\ell y_j y_\ell^{-1})_\infty
(\w t_j^{-1} y_j y_\ell^{-1})_\infty}
\frac{(\w q_j \w t_\ell^{-1} y_\ell y_j^{-1}q)_\infty 
(y_\ell y_j^{-1}q)_\infty}{(\w q_j y_\ell y_j^{-1}q)_\infty 
(\w t_\ell^{-1} y_\ell y_j^{-1}q)_\infty}.
\end{eqnarray}
By \eqref{BB11250927Z.2} and 
\cite[Theorem~1.1 and Theorem~1.3 (ii)]{Qin3}, we obtain
$$
\langle \ch_{k_1} \cdots \ch_{k_N} \rangle_{0, t_1, t_2}' 
\in \qBD(t_1, t_2)
$$ 
whose weight is at most $\sum_{i=1}^N (k_i+2)$. 
By \eqref{Car1(23).1}, since $m=0$, we get $a_{\la, \la} = 1$ for every partition $\la$. By Lemma~\ref{BB11241134Z}, 
$\langle \ch_{k_1} \cdots \ch_{k_N} \rangle_{0, t_1, t_2}'$ 
is a degree-$\sum_{i=1}^N k_i$ homogeneous polynomial in $t_1$ and $t_2$ 
with coefficients in $\Q[[q]]$. So 
$$
\langle \ch_{k_1} \cdots \ch_{k_N} \rangle_{0, t_1, t_2}' 
\in \qBD[t_1, t_2].
$$ 
Since $t_1$ and $t_2$ are symmetric, 
$\langle \ch_{k_1} \cdots \ch_{k_N} \rangle_{0, t_1, t_2}'$ 
is symmetric in $t_1$ and $t_2$.

(ii) We see from \eqref{BB070340X.1} that 
$\langle \ch_{k} \rangle_{0, t_1, t_2}'$ is equal to
$$
t_2^k \cdot \Coe_{z^{k}} \frac{1}{(1-\w q)(1-\w t^{-1})} \left (1
   - \frac{(q)_\infty (\w q \w t^{-1}q)_\infty}
   {(\w q q)_\infty (\w t^{-1}q)_\infty} \right )
$$
where $\w q = e^{\alpha z}$ and $\w t = e^{z}$. Note that
$$
  \frac{1}{1-e^x}
= -\frac1x \left (1 - \frac12 x + \frac1{12}x^2 
  - \frac1{720}x^4+ O(x^5) \right )
\in \Q[[x]].
$$
Also, in view of \cite[Remark~3.2]{Qin3}, we conclude that 
\begin{eqnarray*}
& &1 - \frac{(q)_\infty (\w q \w t^{-1}q)_\infty}
   {(\w q q)_\infty (\w t^{-1}q)_\infty}  \\
&=&-[2] \alpha z^2 - [3] (\alpha - 1)\alpha z^3 
   - [4] \alpha^3 z^4 - \left (\frac12 [2]^2 - \frac32 [4] \right ) 
   \alpha^2 z^4 - [4] \alpha z^4 + O(z^5)  \\
&=&-\alpha z^2 \cdot \sum_{m=0}^\infty h_m(\alpha) z^m
\end{eqnarray*}
where $h_m(\alpha) \in \qMZV[\alpha]$ for each $m \ge 0$, 
$\deg_\alpha h_m(\alpha) = m$, the weight of 
$h_m(\alpha) \in \qMZV[\alpha]$ is $m+2$, and the coefficient of $\alpha^m$ 
in $h_m(\alpha)$ is equal to $[m+2]$. Therefore,
$$
\langle \ch_{k} \rangle_{0, t_1, t_2}' \in \qMZV[t_1, t_2]
$$ 
is a degree-$k$ symmetric homogeneous polynomial in $t_1$ and $t_2$ 
whose coefficients have weights at most $k+2$. Moreover,
the coefficients of $t_1^k$ and $t_2^k$ are of the form 
$$
(-1)^k \cdot [k+2] + \big (\text{terms with weights $< (k+2)$} \big ).
$$

(iii) By \eqref{rmk_trace.0}, 
$\langle \ch_{k_1} \cdots \ch_{k_N} \rangle_{m, t_1, t_2}
= [\ch_{k_1} \cdots \ch_{k_N}]_{m, t_1, t_2}$. By \eqref{ExampleF01.1},
we obtain $\langle \ch_{k_1} \cdots \ch_{k_N} \rangle_{m, t_1, t_2}'
= [\ch_{k_1} \cdots \ch_{k_N}]_{m, t_1, t_2}'$. 
So (iii) follows from Theorem~\ref{BB11250927Z}.
\end{proof}


\section{\bf Equivariant higher order derivatives} 
\label{sect_HigherDeri}

Our goal in this section is to study the higher order derivatives 
$\fa_{n}^{(k)}, \, k \ge 1$ of 
the equivariant Heisenberg operators $\fa_{n}, \, n \ne 0$ 
(see Definition~\ref{Intro_DefEquivaDeri}). The main idea is to utilize 
the formula $\fG_1(t_1, t_2) = t_2^k t_1^{\delta(\cdot)} 
\overline{\fB}_1^{(\alpha)}|_{\alpha = -t_1/t_2}$ from 
Theorem~\ref{Intro_TT10210727X}. This formula interprets 
the equivariant boundary operator 
$\fG_1(t_1, t_2)$ in terms of vertex operators through \eqref{Intro_CW1}, 
\eqref{Intro_TT10180917Z.0} and \eqref{Intro_OverlinefBkAlpha}. 

The first order derivative $\fa_{n}' = \fa_{n}^{(1)}$ is given below.

\begin{lemma}   \label{BB02170918Z}
For nonzero integers $i$ and $j$, define $\epsilon(i, j) = 1$ 
if $i j < 0, i+j > 0$ or if $i, j < 0$, and $\epsilon(i, j) = 0$ otherwise.
Then, for $n \ne 0$, we have 
\begin{eqnarray}   \label{BB02170918Z.0}
\fa_{n}' 
= \frac{n}{2} \sum_{i+j=n} (-t_1t_2)^{\epsilon(i, j)} \fa_{i}\fa_{j}
  + \frac{n(|n|-1)}{2} (t_1+t_2) \fa_{n}.
\end{eqnarray}
\end{lemma}
\begin{proof}
Follows from \eqref{Intro_OP3DefMqt1t2} and 
\eqref{Intro_TT09110234X} via straightforward computations.
\end{proof}

To compute $\fa_{-n}^{(k)}$ for $n, k \ge 1$, we begin with two simple lemmas.

\begin{lemma}  \label{CommuVFa}
Let $n > 0$, and $V(z; q, t, \w q, \w t)$ be from \eqref{CW1}. Then, 
$$
  [V(z; q, t, \w q, \w t), \fa_{-n}] 
= (\w t^n - t^n)z^{-n} \cdot V(z; q, t, \w q, \w t).
$$
\end{lemma}
\begin{proof}
By \eqref{CW1}, $V(z; q, t, \w q, \w t) = \Gamma_-(qz) \Gamma_-(\w qz)^{-1} 
\Gamma_+(\w t^{-1}z) \Gamma_+(t^{-1}z)^{-1}$. So 
\begin{eqnarray*}    
   [V(z; q, t, \w q, \w t), \fa_{-n}]
&=&\Gamma_-(qz) \Gamma_-(\w qz)^{-1} 
   [\Gamma_+(\w t^{-1}z), \fa_{-n}] \Gamma_+(t^{-1}z)^{-1}     \\
& &+\,\, \Gamma_-(qz) \Gamma_-(\w qz)^{-1} 
   \Gamma_+(\w t^{-1}z) [\Gamma_+(t^{-1}z)^{-1}, \fa_{-n}]  \\
&=&(\w t^n - t^n)z^{-n} \cdot V(z; q, t, \w q, \w t)
\end{eqnarray*}
where we have used \eqref{GammaPMComm.01} twice in the last step.
\end{proof}

\begin{lemma}  \label{CommuBwtV}
Let $k \ge 1$. For $1 \le i < j \le k$, define
$$
  f_{i, j} 
= \frac{1-(t_j^{-1}z_j)^{-1}(q_iz_i)}{1-(\w t_j^{-1}z_j)^{-1}(q_iz_i)} \cdot
  \frac{1-(\w t_j^{-1}z_j)^{-1}(\w q_iz_i)}{1-(t_j^{-1}z_j)^{-1}(\w q_iz_i)}.
$$
Then, 
$[V(z_k; q_k, t_k, \w q_k, \w t_k), [\ldots,
[V(z_2; q_2, t_2, \w q_2, \w t_2), V(z_1; q_1, t_1, \w q_1, \w t_1)] \ldots ]$
is equal to
$$
\prod_{j=2}^k \left (\prod_{i=1}^{j-1}f_{i,j} - \prod_{i=1}^{j-1}f_{j, i} \right )   
\cdot \prod_{i=1}^k \Gamma_-(q_iz_i) \Gamma_-(\w q_iz_i)^{-1}
\cdot \prod_{i=1}^k \Gamma_+(\w t_i^{-1}z_i) \Gamma_+(t_i^{-1}z_i)^{-1}.
$$
\end{lemma}
\begin{proof}
First of all, the lemma is trivially true when $k = 1$.
Next, by \eqref{CW1} and Lemma~\ref{GammaPMComm}~(ii), 
$[V(z_2; q_2, t_2, \w q_2, \w t_2), V(z_1; q_1, t_1, \w q_1, \w t_1)]$ is equal to
\begin{eqnarray*}    
& &\Gamma_-(q_2z_2) \Gamma_-(\w q_2z_2)^{-1} 
   \Gamma_+(\w t_2^{-1}z_2) \Gamma_+(t_2^{-1}z_2)^{-1}   \\
& &\cdot \Gamma_-(q_1z_1) \Gamma_-(\w q_1z_1)^{-1}    
   \Gamma_+(\w t_1^{-1}z_1) \Gamma_+(t_1^{-1}z_1)^{-1} \\
& &-\, \Gamma_-(q_1z_1) \Gamma_-(\w q_1z_1)^{-1}    
   \Gamma_+(\w t_1^{-1}z_1) \Gamma_+(t_1^{-1}z_1)^{-1} \\
& &\cdot \Gamma_-(q_2z_2) \Gamma_-(\w q_2z_2)^{-1} 
   \Gamma_+(\w t_2^{-1}z_2) \Gamma_+(t_2^{-1}z_2)^{-1} \\
&=&(f_{1,2}-f_{2,1})   
   \cdot \prod_{i=1}^2 \Gamma_-(q_iz_i) \Gamma_-(\w q_iz_i)^{-1}
   \cdot \prod_{i=1}^2 \Gamma_+(\w t_i^{-1}z_i) \Gamma_+(t_i^{-1}z_i)^{-1}.
\end{eqnarray*}
So the lemma holds for $k = 2$. Using induction on $k$ yields the lemma.
\end{proof}

The following is the main result in this section.

\begin{theorem}  \label{EquivHigherDeriv}
Let $n > 0$ and $k \ge 1$. Let $t_1^{\delta(\cdot)}$ be the operator 
defined by putting
$t_1^{\delta(\cdot)} \fa_\la = t_1^{\ell(\la^-) - \ell(\la^+)} \fa_\la$
for any generalized partition $\la$. Then, $\fa_{-n}^{(k)}$ is equal to
$$
(-t_2)^k t_1^{\delta(\cdot)-1} \Coe_{s_1 \cdots s_k z_1^0 \cdots z_k^0} 
   \frac{(Q_1^n - 1)T_1^{-n}z_1^{-n}}{\prod_{i=1}^k (1-Q_i)(1-T_i^{-1})}
   \prod_{j=2}^k \left (\prod_{i=1}^{j-1}f_{i,j} 
   - \prod_{i=1}^{j-1}f_{j, i} \right )   
$$
$$
\cdot \prod_{i=1}^k \Gamma_-(T_iz_i) \Gamma_-(z_i)^{-1}
\cdot \prod_{i=1}^k \Gamma_+(Q_i^{-1}T_iz_i) \Gamma_+(T_iz_i)^{-1}.
$$
where $Q_i = e^{\alpha s_i}$ with $\alpha = -t_1/t_2$, $T_i = e^{s_i}$, and
\begin{eqnarray}  \label{EquivHigherDeriv.0}
  f_{i, j} 
= \frac{1-(T_jz_j)^{-1}(T_iz_i)}{1-(Q_j^{-1}T_jz_j)^{-1}(T_iz_i)} \cdot
  \frac{1-(Q_j^{-1}T_jz_j)^{-1}z_i}{1-(T_jz_j)^{-1}z_i}.
\end{eqnarray}
\end{theorem}
\begin{proof}
By Theorem~\ref{TT10210727X} and \eqref{Intro_OverlinefBkAlpha}, we have
$$
  \fG_1(t_1, t_2)
= t_2 t_1^{\delta(\cdot)} \overline{\fB}_1^{(\alpha)}
= t_2 t_1^{\delta(\cdot)} \Coe_{t_0} \overline{\fB}(q, t^{-1})
$$
where $q = e^{\alpha t_0}$ with $\alpha = -t_1/t_2$, and $t = e^{t_0}$. 
Combining with \eqref{TT10180917Z.0}, we obtain
\begin{eqnarray*}   
   \fG_1(t_1, t_2)
&=&t_2 t_1^{\delta(\cdot)} \Coe_{t_0} \frac{1}{(1-q)(1-t^{-1})} 
   \big (1 - V_0(z; t, t^{-1}, 1, qt^{-1})\big )     \\
&=&t_2 t_1^{\delta(\cdot)} \Coe_{t_0z^0} \frac{1}{(1-q)(1-t^{-1})} 
   \big (1 - V(z; t, t^{-1}, 1, qt^{-1})\big ).
\end{eqnarray*}
To avoid confusion with the notations, we write the above as 
\begin{eqnarray}  \label{EquivHigherDeriv.1}
  \fG_1(t_1, t_2)
= t_2 t_1^{\delta(\cdot)} \Coe_{sz^0} \frac{1}{(1-Q)(1-T^{-1})} 
   \big (1 - V(z; T, T^{-1}, 1, QT^{-1})\big )
\end{eqnarray}
by replacing $q, t, t_0$ by $Q, T, s$ respectively 
(so $Q = e^{\alpha s}$ and $T = e^{s}$). By definition,
$$
  \fa_{-n}^{(k)}
= \underbrace{[\fG_1(t_1, t_2), [\ldots, [\fG_1(t_1, t_2)}_{\text{$k$ times}}, 
   \fa_{-n}] \ldots ].
$$
Applying \eqref{EquivHigherDeriv.1}, we conclude that $\fa_{-n}^{(k)}$ is equal to
$$
(-t_2)^k t_1^{\delta(\cdot)-1} \Coe_{s_1 \cdots s_k z_1^0 \cdots z_k^0} 
   \frac{1}{\prod_{i=1}^k (1-Q_i)(1-T_i^{-1})}
$$
\begin{eqnarray*}   
[V(z_k; T_k, T_k^{-1}, 1, Q_kT_k^{-1}), 
[\ldots, [V(z_1; T_1, T_1^{-1}, 1, Q_1T_1^{-1}), \fa_{-n}] \ldots ]    
\end{eqnarray*}
where $Q_i = e^{\alpha s_i}$ and $T_i = e^{s_i}$.
By Lemma~\ref{CommuVFa}, $\fa_{-n}^{(k)}$ is equal to
$$
(-t_2)^k t_1^{\delta(\cdot)-1} \Coe_{s_1 \cdots s_k z_1^0 \cdots z_k^0} 
   \frac{(Q_1^n - 1)T_1^{-n}z_1^{-n}}{\prod_{i=1}^k (1-Q_i)(1-T_i^{-1})}
$$
$$
[V(z_k; T_k, T_k^{-1}, 1, Q_kT_k^{-1}), 
[\ldots, [V(z_2; T_2, T_2^{-1}, 1, Q_2T_2^{-1}), 
$$
$$
V(z_1; T_1, T_1^{-1}, 1, Q_1T_1^{-1})] \ldots ].    
$$
Finally, applying Lemma~\ref{CommuBwtV} completes the proof of our lemma.
\end{proof}

Our next result determines the leading term of $\fa_{-n}^{(k)}$ 
where $n > 0$ and $k \ge 1$. 

\begin{proposition}  \label{Lma_EquivHigherDeriv}
Let $n > 0$ and $k \ge 1$. Then, $\fa_{-n}^{(k)}$ is equal to
$$
n^k k! \cdot \sum_{\ell(\la) = k + 1, \, |\la| = -n} (-1)^{\ell(\la^+)} 
(t_1t_2)^{\ell(\la^-)-1} \frac{\fa_\la}{\la^!} 
+ \sum_{\ell(\la) < k + 1, \, |\la| = -n} f^{(k)}_\la \frac{\fa_\la}{\la^!}
$$
where $f^{(k)}_\la \in \Q(t_1, t_2)$, and $\la^\pm$ denote 
the positive and negative parts of $\la$.
\end{proposition}
\begin{proof}
We will compute the terms in the expression of $\fa_{-n}^{(k)}$ from 
Theorem~\ref{EquivHigherDeriv}. Recall that $Q_i = e^{\alpha s_i}$ with 
$\alpha = -t_1/t_2$ and $T_i = e^{s_i}$. We get
\begin{eqnarray}   \label{TT11251039Z}
& &\frac{(Q_1^n - 1)T_1^{-n}z_1^{-n}}{\prod_{i=1}^k (1-Q_i)(1-T_i^{-1})}  
     \nonumber   \\
&=&\frac{(-1)^knz_1^{-n}}{\alpha^{k-1} s_1 \prod_{i=2}^k s_i^2} \cdot 
     \left (1 + \frac12 (\alpha - 2)ns_1 
     - \frac12(\alpha -1)\sum_{i=1}^k s_i \right.     \nonumber   \\
& & + \frac16 (\alpha^2 - 3\alpha + 3)n^2s_1^2 
     - \frac14 (\alpha - 2)(\alpha - 1)ns_1 \sum_{i=1}^k s_i   \nonumber   \\
& &\left. - \frac{1}{24} (\alpha^2+1) \sum_{i=1}^k s_i^2 + \frac18 (\alpha -1)^2 
     \left (\sum_{i=1}^k s_i \right)^2 + O(3) \right )
\end{eqnarray}
where $O(n)$ denotes terms whose combined degrees in $s_1, \ldots, s_k$ are 
at least $n$.

Next, by \eqref{Intro_GammaLz}, 
$\prod_{i=1}^k \Gamma_-(T_iz_i) \Gamma_-(z_i)^{-1}
\cdot \prod_{i=1}^k \Gamma_+(Q_i^{-1}T_iz_i) \Gamma_+(T_iz_i)^{-1}$ equals
\begin{eqnarray} \label{TT12190950Z}
\prod_{r > 0} \exp\left (\frac{\sum_{i=1}^k (T_i^r -1)z_i^r}{r} \fa_{-r} 
     \right ) \cdot \prod_{r > 0} \exp\left (\frac{\sum_{i=1}^k 
     (Q_i^rT_i^{-r} -T_i^{-r})z_i^{-r}}{r} \fa_{r}\right ).
\end{eqnarray}
Since $Q_i = e^{\alpha s_i}$ and $T_i = e^{s_i}$,
the above expression is equal to
$$
\prod_{r > 0} \exp\left (\sum_{i=1}^k s_i\left (1 + \frac12 rs_i + 
     \frac16 r^2s_i^2 + O(s_i^3) \right )z_i^r \fa_{-r} \right )   
$$
$$
\cdot \prod_{r > 0} \exp\left (\sum_{i=1}^k s_i\left (\alpha + 
     \frac12(\alpha^2 - 2\alpha)rs_i + \frac16(\alpha^3 - 3\alpha^2 
     + 3\alpha)r^2s_i^2 + O(s_i^3)\right )z_i^{-r} \fa_{r} \right ).
$$
Thus, $\prod_{i=1}^k \Gamma_-(T_iz_i) \Gamma_-(z_i)^{-1}
\cdot \prod_{i=1}^k \Gamma_+(Q_i^{-1}T_iz_i) \Gamma_+(T_iz_i)^{-1}$ is equal to
\begin{eqnarray} \label{TT11250422X}
& &\sum_\la \frac{\fa_\la}{\la^!} \cdot \alpha^{\ell(\la^+)}   
     \cdot \left ( \prod_{r > 0} \left (\left (
     \sum_{i=1}^k s_i z_i^r\right)^{m_{-r}} \right. \right.   \nonumber  \\
& &+ \frac{rm_{-r}}{2} \left (\sum_{i=1}^k s_i z_i^r\right)^{m_{-r}-1} \cdot
     \sum_{i=1}^k\left (s_i^2 + \frac13 rs_i^3 \right )z_i^r   \nonumber   \\
& &\left. + \frac{r^2}{4} \binom{m_{-r}}{2} 
     \left (\sum_{i=1}^k s_i z_i^r\right)^{m_{-r}-2} \cdot
     \left (\sum_{i=1}^k s_i^2z_i^r \right )^2 \right )  
     \cdot \prod_{r > 0} \left (\left (\sum_{i=1}^k s_i z_i^{-r} \right)^{m_{r}}  
     \right.   \nonumber   \\
& &+ \frac{rm_{r}}{2} \left (\sum_{i=1}^k s_i z_i^{-r} \right)^{m_{r}-1} \cdot
     \sum_{i=1}^k\left ((\alpha - 2)s_i^2 
     + \frac{\alpha^2-3\alpha + 3}{3} rs_i^3 \right )z_i^{-r}   \nonumber   \\
& &\left. \left. + \frac{r^2}{4} \binom{m_{r}}{2} 
     \left (\sum_{i=1}^k s_i z_i^{-r}\right)^{m_{r}-2} \cdot
     \left (\sum_{i=1}^k (\alpha - 2)s_i^2z_i^{-r} \right )^2 \right )
     + O(\ell(\la) + 3) \right )
\end{eqnarray}
where $\la$ denotes a generalized partition whose multiplicity of part $n$ 
is $m_n$.

To calculate $f_{i, j}$ with $i \ne j$, we put $z_{i, j} = z_iz_j^{-1}$ 
for simplicity. By \eqref{EquivHigherDeriv.0},
\begin{eqnarray} \label{TT11250424X.1}
  f_{i, j} 
= \frac{1-e^{s_i-s_j}z_{i, j}}{1-e^{s_i+(\alpha-1)s_j}z_{i, j}} \cdot
  \frac{1-e^{(\alpha-1)s_j}z_{i, j}}{1-e^{-s_j}z_{i, j}}.
\end{eqnarray}
The numerator $(1-e^{s_i-s_j}z_{i, j})(1-e^{(\alpha-1)s_j}z_{i, j})$ is equal to
\begin{eqnarray} \label{TT11250424X.2}
& &1 - e^{s_i-s_j}z_{i, j} - e^{(\alpha-1)s_j}z_{i, j} + 
     e^{s_i+(\alpha-2)s_j}z_{i, j}^2    \nonumber   \\
&=&(1 - z_{i, j})^2 - \sum_{\ell = 1}^4 \frac{1}{\ell!}(s_i-s_j)^\ell z_{i, j}
     - \sum_{\ell = 1}^4 \frac{1}{\ell!}((\alpha-1)s_j)^\ell z_{i, j}  
     \nonumber   \\
& &\qquad + \, \sum_{\ell = 1}^4 
     \frac{1}{\ell!}(s_i+(\alpha-2)s_j)^\ell z_{i, j}^2 + O(5).
\end{eqnarray}
For the denominators in \eqref{TT11250424X.1}, 
$1/(1-e^{s_i+(\alpha-1)s_j}z_{i, j})$ is equal to
$$
\frac{1}{1-z_{i,j}} + \frac{z_{i,j}}{(1-z_{i,j})^2} (s_i+(\alpha-1)s_j)
   + \frac1{2!} \frac{z_{i,j} + z_{i,j}^2}{(1-z_{i,j})^3} (s_i+(\alpha-1)s_j)^2
$$
$$
+ \frac1{3!} \frac{z_{i,j} + 4z_{i,j}^2+z_{i,j}^3}{(1-z_{i,j})^4} 
    (s_i+(\alpha-1)s_j)^3
$$
\begin{eqnarray}   \label{TT11250424X.3}
+ \frac1{4!} \frac{z_{i,j} + 11 z_{i,j}^2 + 11z_{i,j}^3 + z_{i,j}^4}{(1-z_{i,j})^5} 
    (s_i+(\alpha-1)s_j)^4 + O(5).
\end{eqnarray}
Similarly, we conclude that $1/(1-e^{-s_j}z_{i, j})$ is equal to
$$
\frac{1}{1-z_{i,j}} + \frac{z_{i,j}}{(1-z_{i,j})^2} (-s_j)
   + \frac1{2!} \frac{z_{i,j} + z_{i,j}^2}{(1-z_{i,j})^3} (-s_j)^2
$$
\begin{eqnarray}   \label{TT11250424X.4}
+ \frac1{3!} \frac{z_{i,j} + 4z_{i,j}^2+z_{i,j}^3}{(1-z_{i,j})^4} 
    (-s_j)^3
+ \frac1{4!} \frac{z_{i,j} + 11 z_{i,j}^2 + 11z_{i,j}^3 + z_{i,j}^4}{(1-z_{i,j})^5} 
    (-s_j)^4 + O(5).
\end{eqnarray}
Combining \eqref{TT11250424X.1}, \eqref{TT11250424X.2}, 
\eqref{TT11250424X.3} and \eqref{TT11250424X.4}, 
we see that $f_{i, j}$ is equal to
$$
1 + \frac{\alpha z_{i,j}}{(1-z_{i,j})^2} s_is_j 
+ \frac{\alpha z_{i,j}(1+z_{i,j})}{2(1-z_{i,j})^3} s_i^2s_j 
+ \frac{\alpha(\alpha-2) z_{i,j}(1+z_{i,j})}{2(1-z_{i,j})^3} s_is_j^2
$$
$$
+ \frac{\alpha z_{i,j}(1+4z_{i,j}+z_{i,j}^2)}{6(1-z_{i,j})^4} s_i^3s_j 
+ \frac{\alpha z_{i,j}(-2(1+4z_{i,j}+z_{i,j}^2) + \alpha(1+6z_{i,j}+z_{i,j}^2))}{4(1-z_{i,j})^4} s_i^2s_j^2
$$
\begin{eqnarray}   \label{TT11250424X.5}
+ \frac{\alpha(\alpha^2-3\alpha+3) z_{i,j}(1+4z_{i,j}+z_{i,j}^2)}
  {6(1-z_{i,j})^4} s_is_j^3 + \alpha s_is_j O(3).
\end{eqnarray}
noting that $f_{i, j} -1$ is divisible by $\alpha s_is_j$. It follows that 
\begin{eqnarray}   \label{TT11250424X.6}
& &\prod_{j=2}^k \left (\prod_{i=1}^{j-1}f_{i,j} 
   - \prod_{i=1}^{j-1}f_{j, i} \right )  \nonumber   \\
&=&\alpha^{k-1} \prod_{j=2}^k \sum_{i=1}^{j-1} s_is_j
   \left (\frac{z_{i,j}}{(1-z_{i,j})^2}  
   + \frac{z_{i,j}(1+z_{i,j})}{2(1-z_{i,j})^3} s_i 
   + \frac{(\alpha-2) z_{i,j}(1+z_{i,j})}{2(1-z_{i,j})^3} s_j
   \right.   \nonumber  \\
& &\left. - \frac{z_{j, i}}{(1-z_{j, i})^2} 
   - \frac{z_{j, i}(1+z_{j, i})}{2(1-z_{j, i})^3} s_j 
   - \frac{(\alpha-2) z_{j, i}(1+z_{j, i})}{2(1-z_{j, i})^3} s_i 
    + O(2) \right ). 
\end{eqnarray}

We see from Theorem~\ref{EquivHigherDeriv}, \eqref{TT11251039Z},
\eqref{TT11250422X} and \eqref{TT11250424X.6} that 
$\fa_{-n}^{(k)}$ is equal to
$$
(-1)^kt_2^k t_1^{\delta(\cdot)-1} \cdot \Coe_{s_1 \cdots s_k z_1^0 \cdots z_k^0} 
   \frac{(-1)^knz_1^{-n}}{\alpha^{k-1} s_1 \prod_{i=2}^k s_i^2} 
$$
$$
\cdot \left (1 + \frac12 (\alpha - 2)ns_1 
- \frac12(\alpha -1)\sum_{i=1}^k s_i + O(2) \right )
$$
$$
\cdot 
\alpha^{k-1} \prod_{j=2}^k \sum_{i=1}^{j-1} s_is_j
   \left (\frac{z_{i,j}}{(1-z_{i,j})^2}  
   + \frac{z_{i,j}(1+z_{i,j})}{2(1-z_{i,j})^3} s_i 
   + \frac{(\alpha-2) z_{i,j}(1+z_{i,j})}{2(1-z_{i,j})^3} s_j
   \right.
$$
$$
\left. - \frac{z_{j, i}}{(1-z_{j, i})^2} 
   - \frac{z_{j, i}(1+z_{j, i})}{2(1-z_{j, i})^3} s_j 
   - \frac{(\alpha-2) z_{j, i}(1+z_{j, i})}{2(1-z_{j, i})^3} s_i 
    + O(2) \right )
$$
$$
\cdot \sum_\la \frac{\fa_\la}{\la^!} \cdot \alpha^{\ell(\la^+)}   
     \cdot \left ( \prod_{r > 0} \left (\left (
     \sum_{i=1}^k s_i z_i^r\right)^{m_{-r}}    
+ \frac{rm_{-r}}{2} \left (\sum_{i=1}^k s_i z_i^r\right)^{m_{-r}-1} \cdot
     \sum_{i=1}^k s_i^2 z_i^r \right ) \right.      
$$
$$
\left. \cdot \prod_{r > 0} \left (\left (\sum_{i=1}^k s_i z_i^{-r} \right)^{m_{r}}       
+ \frac{rm_{r}}{2} \left (\sum_{i=1}^k s_i z_i^{-r} \right)^{m_{r}-1} 
     \sum_{i=1}^k (\alpha - 2)s_i^2z_i^{-r} \right )
     +O(\ell(\la) + 2) \right )
$$
where $z_{i, j} = z_iz_j^{-1}$ for $i \ne j$, and $\la$ denotes 
a generalized partition whose multiplicity of part $n$ is $m_n$. 
Note that for a fixed generalized partition $\la$, 
the sum of the exponents of $s_1, \ldots, s_k$ 
in the above expression is at least 
$$
-1 - 2(k-1) + 2(k-1) + \ell(\la) = \ell(\la) - 1.
$$
Since $s_1 \cdots s_k$ has degree $k$, the generalized partitions $\la$ with 
$\ell(\la) \ge k + 2$ do not contribute to $\fa_{-n}^{(k)}$. 
Similarly, by considering the exponents of $z_1, \ldots, z_k$,
we see that if $|\la| \ne -n$, then $\la$ does not contribute to $\fa_{-n}^{(k)}$. 
Thus, $\fa_{-n}^{(k)}$ is equal to
$$
\sum_{\ell(\la) \le k + 1, \, |\la| = -n} \frac{\fa_\la}{\la^!} \cdot 
(-1)^{\ell(\la^+)} t_1^{\ell(\la^-)-1} t_2^{k - \ell(\la^+)}      
\cdot \Coe_{s_1 \cdots s_k z_1^0 \cdots z_k^0} 
   \frac{nz_1^{-n}}{\prod_{i=1}^k s_i} 
$$
$$
\cdot \left (1 + \frac12 (\alpha - 2)ns_1 
- \frac12(\alpha -1)\sum_{i=1}^k s_i + O(2) \right )
$$
$$
\cdot \prod_{j=2}^k \sum_{i=1}^{j-1} s_i
   \left (\frac{z_{i,j}}{(1-z_{i,j})^2}  
   + \frac{z_{i,j}(1+z_{i,j})}{2(1-z_{i,j})^3} s_i 
   + \frac{(\alpha-2) z_{i,j}(1+z_{i,j})}{2(1-z_{i,j})^3} s_j
   \right.
$$
$$
\left. - \frac{z_{j, i}}{(1-z_{j, i})^2} 
   - \frac{z_{j, i}(1+z_{j, i})}{2(1-z_{j, i})^3} s_j 
   - \frac{(\alpha-2) z_{j, i}(1+z_{j, i})}{2(1-z_{j, i})^3} s_i 
    + O(2) \right )
$$
$$
\cdot \left (O(\ell(\la) + 2) + \prod_{r > 0} \left (\left (
     \sum_{i=1}^k s_i z_i^r\right)^{m_{-r}}    
+ \frac{rm_{-r}}{2} \left (\sum_{i=1}^k s_i z_i^r\right)^{m_{-r}-1} \cdot
     \sum_{i=1}^k s_i^2 z_i^r \right ) \right.      
$$
$$
\left. \cdot \prod_{r > 0} \left (\left (\sum_{i=1}^k s_i z_i^{-r} \right)^{m_{r}}       
+ \frac{rm_{r}}{2} \left (\sum_{i=1}^k s_i z_i^{-r} \right)^{m_{r}-1} \cdot
     \sum_{i=1}^k (\alpha - 2)s_i^2z_i^{-r} \right )
     \right ).
$$
It follows that $\fa_{-n}^{(k)}$ is of the form
\begin{eqnarray}   \label{fk0Lambda}
\sum_{\ell(\la) = k + 1, \, |\la| = -n} (-1)^{\ell(\la^+)} 
(t_1t_2)^{\ell(\la^-)-1} f^{(k)}_{0, \la} \frac{\fa_\la}{\la^!} 
+ \sum_{\ell(\la) < k + 1, \, |\la| = -n} f^{(k)}_\la \frac{\fa_\la}{\la^!}
\end{eqnarray}
where $f^{(k)}_\la \in \Q(t_1, t_2)$, and $f^{(k)}_{0, \la} \in \Q$ with 
$\ell(\la) = k + 1$ and $|\la| = -n$ is equal to
$$
\Coe_{s_1 \cdots s_k z_1^0 \cdots z_k^0} 
   \frac{nz_1^{-n}}{\prod_{i=1}^k s_i} 
\, \prod_{j=2}^k \sum_{i=1}^{j-1} s_i
  \left (\frac{z_{i,j}}{(1-z_{i,j})^2} - \frac{z_{j, i}}{(1-z_{j, i})^2} \right )
  \cdot \prod_{r \ne 0} \left (\sum_{i=1}^k s_i z_i^r\right)^{m_{-r}}. 
$$

It is very difficult to evaluate $f^{(k)}_{0, \la}$ directly 
via the previous line. Alternatively, by using induction on $k$, 
we now prove that if $\ell(\la) = k + 1$ and $|\la| = -n$, then
\begin{eqnarray}   \label{fk0Lambda.1} 
f^{(k)}_{0, \la} = n^k k!.
\end{eqnarray}
When $k = 0$, we have $\fa_{-n}^{(0)} = \fa_{-n}$.
Comparing with \eqref{fk0Lambda}, we get $f^{(k)}_{0, \la} = 1$ 
for the partition $\la = (-n)$. So \eqref{fk0Lambda.1} is true for $k = 0$.
Now assume that \eqref{fk0Lambda.1} holds for some $k \ge 0$. 
We will show that \eqref{fk0Lambda.1} holds for $k+1$ as well. Indeed, 
by \eqref{fk0Lambda} and our inductive assumption, 
$\fa_{-n}^{(k)}$ is equal to
$$
\sum_{\ell(\la) = k + 1, \, |\la| = -n} (-1)^{\ell(\la^+)} 
(t_1t_2)^{\ell(\la^-)-1} n^k k! \frac{\fa_\la}{\la^!} + 
\sum_{\ell(\la) < k + 1, \, |\la| = -n} f^{(k)}_\la \frac{\fa_\la}{\la^!}.
$$
Putting $t_1=-1$ and $t_2 = 1$,
since $\fa_{-n}^{(k+1)} = \big (\fa_{-n}^{(k)} \big )'$, we see that
$\fa_{-n}^{(k+1)}$ is equal to
\begin{eqnarray}   \label{fk0Lambda.2}
& &\sum_{\ell(\la) = k + 1, \, |\la| = -n} 
   (-n)^k k! \frac{(\fa_\la)'}{\la^!} 
   + \sum_{\ell(\la) < k + 1, \, |\la| = -n} 
   f^{(k)}_\la \frac{(\fa_\la)'}{\la^!}   \nonumber  \\
&=&\sum_{\ell(\mu) = k + 2, \, |\mu| = -n} (-1)^{k+1} 
   f^{(k+1)}_{0, \mu} \frac{\fa_\mu}{\mu^!} 
   + \sum_{\ell(\mu) < k + 2, \, |\mu| = -n} 
   f^{(k+1)}_\mu \frac{\fa_\mu}{\mu^!}
\end{eqnarray}
where by abusing notations, we let $f^{(k)}_\la = f^{(k)}_\la(-1, 1)$ and 
$f^{(k+1)}_\mu = f^{(k+1)}_\mu(-1, 1)$. By \eqref{BB02170918Z.0},
$\fa_m' = {m}/{2} \cdot \sum_{i+j=m} \fa_{i}\fa_{j}$. 
Using the arguments in the proof of \cite[Lemma~3.29]{Qin1}, 
we conclude that $f^{(k+1)}_{0, \mu} = n^{k+1}(k+1)!$ 
whenever $\ell(\mu) = k + 2$ and $|\mu| = -n$. 
Hence \eqref{fk0Lambda.1} is true for $k+1$.

Finally, combining \eqref{fk0Lambda} and \eqref{fk0Lambda.1} completes 
the proof of our proposition. 
\end{proof}

\end{document}